\documentclass{amsart}
\usepackage{graphicx}
\usepackage[dvips,final]{epsfig}
\usepackage{amsmath}

\usepackage{amsfonts}

\usepackage{amssymb}

\usepackage{latexsym}

\newtheorem{thm}{Theorem}[section]
\newtheorem{cor}[thm]{Corollary}
\newtheorem{lem}[thm]{Lemma}
\newtheorem{prop}[thm]{Proposition}

\theoremstyle{definition}

\newtheorem{defn}[thm]{Definition}
\newtheorem{ex}[thm]{Example}
\newtheorem{remark}[thm]{Remark}
\newtheorem{rem}[thm]{Remark}

\numberwithin{equation}{section}

\newcommand{\To}{\longrightarrow}

\newcommand{\bC}{\mathbb C}

\newcommand{\bK}{\mathbb K}

\newcommand{\bR}{\mathbb R}

\newcommand{\bZ}{\mathbb Z}

\newcommand{\de}{\delta}

\newcommand{\g}{\gamma}

\newcommand{\al}{\alpha}

\newcommand{\be}{\beta}

\newcommand{\om}{\omega}

\newcommand{\cA}{\mathcal A}

\newcommand{\cC}{\mathcal C}

\newcommand{\cD}{\mathcal D}

\newcommand{\cE}{\mathcal E}

\newcommand{\cI}{\mathcal I}
\newcommand{\cJ}{\mathcal J}

\newcommand{\cK}{\mathcal K}

\newcommand{\cL}{\mathcal L}

\newcommand{\cM}{\mathcal M}

\newcommand{\cO}{\mathcal O}

\newcommand{\cR}{\mathcal R}

\newcommand{\Id}{\rm Id}

\newcommand{\cod}{{\rm cod}}

\newcommand{\Div}{{\rm div}}
\newcommand{\trace}{{\rm trace}}

\newcommand{\Diff}{{\rm Diff}}

\newcommand{\Derlog}{{\rm Derlog}}

\newcommand{\im}{{\rm im}}

\newcommand{\VqG}{{G}_{\Omega _q}}

\newcommand{\VnG}{{G}_{\Omega _n}}
\newcommand{\VpG}{{G}_{\Omega _p}}

\newcommand{\cVnA}{{\mathcal A}_{\Omega _n}}
\newcommand{\cVpA}{{\mathcal A}_{\Omega _p}}

\newcommand{\cVnK}{{\mathcal K}_{\Omega _n}}
\newcommand{\cVpK}{{\mathcal K}_{\Omega _p}}

\newcommand{\cVnR}{{\mathcal R}_{\Omega _n}}

\newcommand{\cVpL}{{\mathcal L}_{\Omega _p}}

\newcommand{\cVpC}{{\mathcal C}_{\Omega _p}}

\newcommand{\bx}{{\bf x}}

\begin{document}

\title[Volume preserving subgroups of $\cA$ and $\cK$]
{Volume preserving subgroups of $\cA$ and $\cK$ and singularities
in unimodular geometry}

\author{W. Domitrz}
\address{Faculty of Mathematics and Information Science, Warsaw
University of Technology, Pl. Politechniki 1, 00-661 Warszawa,
Poland} \email{domitrz@alpha.mini.pw.edu.pl}

\author{J. H. Rieger}
\address{Institut f{\"u}r Mathematik, Universit{\"a}t
Halle, D-06099 Halle (Saale), Germany }
\email{rieger@mathematik.uni-halle.de}




\maketitle

\begin{abstract}
For a germ of a smooth map $f$ from $\bK ^n$ to $\bK ^p$ and a
subgroup $\VqG$ of any of the Mather groups $G$ for which the
source or target diffeomorphisms preserve some given volume form
$\Omega _q$ in $\bK ^q$ ($q=n$ or $p$) we study the $\VqG$-moduli
space of $f$ that parameterizes the $\VqG$-orbits inside the
$G$-orbit of $f$. We find, for example, that this moduli space
vanishes for $\VqG =\cVpA$ and $\cA$-stable maps $f$ and for $\VqG
=\cVnK$ and $\cK$-simple maps $f$. On the other hand, there are
$\cA$-stable maps $f$ with infinite-dimensional $\cVnA$-moduli
space.
\end{abstract}

\section*{Introduction}\label{intro}

We are going to study singularities arising in unimodular
geometry. A singular subvariety of a space with a fixed volume
form may be given by some parametrization or by defining
equations. This leads to the following (multi-)local
classification problems. (1) The classification of germs of smooth
maps $f:(\bK ^n,0)\to (\bK ^p,\Omega _p,0)$ ($\bK =\bC$ or $\bR$)
up to $\cVpA$-equivalence (i.e., for the subgroup of $\cA$ in
which the left coordinate changes preserve a given volume form
$\Omega _p$ in the target), and also of multi-germs of such maps
up to $\cVpA$-equivalence. (2) The classification of variety-germs
$V=f^{-1}(0)\subset (\bK ^n,\Omega _n,0)$ up to
$\cVnK$-equivalence of $f:(\bK ^n,\Omega _n,0)\to (\bK ^p,0)$
(i.e., for the subgroup of $\cK$ in which the right coordinate
changes preserve a given volume form $\Omega _n$ in the source).
More generally, we will consider volume preserving subgroups
$\VqG$ of any of the Mather groups $G=\cA$, $\cK$, $\cL$, $\cR$
and $\cC$ preserving a (germ of a) volume form $\Omega _q$ in the
source (for $q=n$) or target (for $q=p$). (See the survey \cite{Wa}
for a discussion of the groups $G$ and their tangent spaces $LG$,
or see the beginning of \S \ref{moduli} below for a brief reminder.)

These subgroups $\VqG$ of $G$ fail to be geometric subgroups of
$\cA$ and $\cK$ in the sense of Damon \cite{Da,Da2}, hence the
usual determinacy and unfolding theorems do not hold for $\VqG$.
In this situation moduli and even functional moduli often appear
already in codimension zero, and e.g. for $\cVnR$ this is indeed
the case: a Morse function has a functional modulus (and hence
infinite modality) in the volume preserving case \cite{V}. Hence
it might appear surprising that Martinet wrote 30 years ago in his
book (see p. 50 of the English translation \cite{Mar2}) on the
$\cVpA$ classification problem in unimodular geometry that the
groups involved ``are big enough that there is still some hope of
finding a reasonable classification theorem''. It turns out that
Martinet was right -- the results of this paper imply, for
example, that over $\bC$ the classifications of stable map-germs
for $\cVpA$ and for $\cA$ agree, and hence Mather's \cite{MaVI}
nice pairs of dimensions $(n,p)$. Furthermore, the classifications
of simple complete intersection singularities agree for $\cVnK$
and for $\cK$. Over $\bR$ a $G$-orbit ($G=\cA$ or $\cK$)
corresponds to one or two orbits in the volume preserving (hence
orientation preserving) case, otherwise the results are the same.

We will now summarize our main results. For any of the above Mather
groups $G$, let $G_f$ denote the stabilizer of a map-germ $f$ in
$G$ and let $G_e$, as usual, denote the extended pseudo group
of non-origin preserving diffeomorphisms. The differential
of the orbit map of $f$ (sending $g\in G$ to $g\cdot f$)
defines a map $\g _f: LG\to LG\cdot f$ with kernel $LG_f$.
Let $LG^q_f$ be the projection of $LG_f$ onto the source
(for $q=n$) or the target factor (for $q=p$).
Notice that, for example, the group $G=\cR$ can be viewed
as a subgroup $\cR\times 1$ of $\cA$ with Lie algebra
$L\cR\oplus 0$ -- allowing such trivial factors $1$ enables us
to define the projections $LG^q_f$ for all Mather groups $G$, which
will be convenient for the uniformness of the exposition.
For a given volume
form $\Omega _q$ in $(\bK ^q,0)$ we have a map
$\Div :\cM _q\cdot\theta _q\to C_r$ sending a vector field (vanishing
at 0) to its divergence, where $r=q$ for all $\VqG$ except
$\cVpK$ (we use here the following standard notation: $C_q$ denotes
the local ring of smooth function germs on $(\bK ^q,0)$ with maximal
ideal $\cM _q$, and $\theta _q$ denotes the $C_q$ module of vector
fields on $(\bK ^q,0)$). For $\cVpK$ we consider linear vector
fields in $(\bK ^p,0)$ with coefficients in $C_n$, the divergence
of such a vector field is an element of $C_n$.
We will show that for the (infinitesimal)
$\VqG$ moduli space $\cM (\VqG ,f)$ we have the following
isomorphism
\[
\cM (\VqG ,f):=\frac{LG\cdot f}{L\VqG\cdot f}
\cong
\frac{C_r}{\Div (LG_f^q)}.
\]
For $\cVnK$ the vector space on the right is in turn isomorphic
to the $n$th cohomology group of a certain subcomplex of the
de Rham complex associated with any finitely generated ideal $\cI$ in
$C_n$ (defined in Section \ref{cohom}), taking
$\cI =\langle f_1,\ldots ,f_p\rangle$ (the ideal generated
by the component functions $f_i$ of $f$).
For $\cVpA$ we obtain an analogous isomorphism by
taking the vanishing ideal $\cI$ of the discriminant (for $n\ge p$)
or the image (for $n<p$) of $f$, provided $L\cA _f^p$ (also known
as ${\rm Lift}(f)$) is equal to $\Derlog$ of the discriminant or image
of $f$.

Furthermore, if $LG$ has the structure of a $C_r$-module (this is
the case for all $\VqG$ except $\cVnA$) then $\dim \cM (\VqG ,f)$
is equal to the number of $\VqG$ moduli of $f$ (for $\cVnA$ this
equality becomes a lower bound). This will be shown in the
following way. The notion of $\VqG$-equivalence of maps $f$ and
$g$ (for a given volume form $\Omega _q$) is easily seen to be
equivalent to the following notion of $G_f^q$-equivalence of
volume forms $\Omega _q$ and $\Omega '_q$ (for a given map $f$):
$\Omega '_q\sim _{G_f^q}\Omega _q$ if and only if for some $h\in
G_f^q$ we have that $h^{\ast}\Omega '_q=\Omega _q$. It then turns
out that a pair $\Omega _q$ and $\Omega '_q$ (that in the case of
$\bR$ defines the same orientation) can be joined by a path of
$G_f^q$ equivalent volume forms if and only if $\Omega '_q-\Omega
_q=d(\xi \rfloor\Omega )$ for some $\xi\in LG_f^q$ and any volume
form $\Omega$ in $(\bK ^q,0)$. And the number of $G_f^q$ moduli of
volume forms (and hence of $\VqG$ moduli of $f$) is given by the
dimension of the space ${\Lambda^q}/\{ d(\xi \rfloor\Omega
):\xi\in LG_f^q\}$ (here $\Lambda^q$ denotes the space of
$q$-forms in $(\bK ^q,0)$), which turns out to be equal to $\dim
C_q/\Div (LG_f^q)$.

If, furthermore, $\cM (\VqG ,f)=0$ then, over $\bC$,
we have at the formal level (and also in the smooth
category, provided the sufficient vanishing condition w.q.h.
for $\cM (\VqG ,f)$ below holds)
\[
\VqG \cdot f=G\cdot f
\]
Over $\bR$, the orbit $G\cdot f$ consists of one or two
$\VqG$-orbits, due to orientation as mentioned above.
More precisely, if $G^+$ denotes the subgroup of $G$ for
which the elements of the $q$-factor of $G$ are
orientation-preserving then $\VqG \cdot f=G^+\cdot f$.

For the most interesting groups $\VqG$ we have the following
sufficient conditions for the vanishing of $\cM (\VqG ,f)$,
namely certain weak forms of quasihomogeneity. We call $f$
{\em weakly quasihomogeneous for} $\VqG$ if $f$ is q.h. for
weights $w_i\in\bZ$ and weighted degrees $\de _j$ such that
the following conditions hold.
\begin{itemize}
\item For $\VqG =\cVpA$: all $\de _j\ge 0$ and $\sum _j \de _j>0$.
\item For $\VqG =\cVnK$: all $w _i\ge 0$ and $\sum _i w_i>0$.
\item For $\VqG =\cVpK$: $\sum _j \de _j\neq 0$.
\end{itemize}
Notice that any $f$ with some zero component function (up
to the relevant $G$-equivalence) is w.q.h. for $\cVpA$ and
$\cVpK$ (and also for $\cVpL$ and $\cVpC$), and any $f$ such that
$df(0)$ has positive rank is w.q.h. for $\cVnK$ and $\cVpK$.
These ``trivial forms of weak quasihomogeneity'' correspond to
the fact that diffeomorphisms of a proper submanifold in
$(\bK ^q,0)$ can be extended to volume preserving diffeomorphisms
of $(\bK ^q,\Omega _q,0)$. Furthermore, if $f$ is
$\VqG$-w.q.h. then the statement about equality of $G$-
and $\VqG$-orbits over $\bC$ (and the corresponding one
over $\bR$) in the previous paragraph holds in the smooth
category (where smooth means complex-analytic over $\bC$ and
$C^{\infty}$ or real-analytic over $\bR$, as usual).
For a $\VqG$-w.q.h. map $f$ the above (generalized) weights
and weighted degrees yield a generalized Euler vector field
in $(\bK ^q,0)$ ($q=n$ or $p$) that allows us to integrate
the (a priori formally defined) vector fields at the infinitesimal
level to give the required smooth diffeomorphisms.

For $f$ not $\VqG$-w.q.h. we are interested in upper and
lower bounds for the dimension of $\cM (\VqG ,f)$ and in
the question whether the $G$-finiteness of $f$ implies
the finiteness of $\cM (\VqG ,f)$. We have several results
in this direction.

(1) For any $\VqG$ for which there is a version of weak
quasihomogeneity we have the following easy upper bound (in the
formal category) for $G$-semiquasihomogeneous (s.q.h.) maps
$f=f_0+h$, where $f_0$ q.h. (and hence $\VqG$-w.q.h.) and
$G$-finite and $h$ has positive degree (relative to the weights of
$f_0$). The normal space $NG\cdot f_0:=\cM _n\cdot\theta
_{f_0}/LG\cdot f_0$ (where $\theta _{f_0}$ denotes the
$C_n$-module of sections of $f_0^{\ast}T\bK ^p$) decomposes into a
part of non-positive filtration and a part of positive filtration,
denoted by $(NG\cdot f_0)_+$. Denoting the number of $G$-moduli of
positive filtration of $f$ by $m(G,f)$ we have the inequality
\[
\dim\cM (\VqG ,f)+m(G,f)\le \dim (NG\cdot f_0)_+.
\]
(Note that the same inequality holds for the
extended pseudo-groups $G_e$, $G_{\Omega _{q},e}$.)
For $\VqG =\cVpA$
all our examples support the following conjecture: for $f$ as above,
the upper bound is actually an equality.
For $\cA$-s.q.h. map-germs
$f:(\bK ^n,0)\to (\bK ^p,\Omega _p,0)$ with $n\ge p-1$ and
$(n,p)$ in the nice range of dimensions or of corank one
(outside the nice range) the validity of this conjecture
would have an interesting consequence. Following Damon and Mond \cite{DM}
we denote by $\mu _{\Delta}(f)$ the discriminant (for $n\ge p$)
or image (for $p=n+1$) Milnor number of $f$ (the discriminants
and images $\Delta (f)$ in these dimensions are hypersurfaces in the target,
and $\Delta (f_t)$ of a stable perturbation $f_t$ of $f$ has
the homotopy type of a wedge of $\mu _{\Delta}(f)$ spheres).
For a q.h. map-germ $f_0$ we have $\cod (\cA _e,f_0)=\mu _{\Delta}(f_0)$
for $n\ge p$ by the main result in \cite{DM} and for
$p=n+1$ by Mond's conjecture (see Conjecture I in \cite{CMWA},
for $n=1,2$ this conjecture has been proved by Mond and others).
Now if our conjecture is true we obtain for s.q.h. maps
$f=f_0+h$ the following interesting consequence of these results:
\[
\cod (\cA _{\Omega _p,e},f)=\mu _{\Delta}(f).
\]
For $(n,p)=(1,2)$ the invariant $\mu _{\Delta}(f)$
is just the classical $\delta$-invariant, hence we recover
the formula $\cod (\cA _{\Omega _p,e},f)=\delta (f)$ of Ishikawa
and Janeczko \cite{IJ} in the special case of s.q.h. curves
(their formula holds for any $\cA$-finite curve-germ).
Notice that for $f=f_0+h$ we have $\mu _{\Delta}(f)=\mu _{\Delta}(f_0)$
(because any deformation by terms of positive filtration is
topologically trivial). Our conjecture implies that the coefficients
of each of the $\dim (N\cA _e\cdot f_0)_+$ terms of $h$ are moduli
for $\cA _{\Omega _p,e}$ (some of them may be moduli for $\cA _e$ too),
hence $\cod (\cA _{\Omega _p,e},f)=\cod (\cA _e,f_0)=\mu _{\Delta}(f_0)$,
which gives the formula above.

(2) For $\VqG =\cVnK$ we have more general results (in the analytic
category) which, for example, imply the following. For any
$\cK$-finite map $f$ the moduli space $\cM (\cVnK ,f)$ is
finite dimensional. Furthermore, if $f^{-1}(0)$ lies
in a hypersurface $h^{-1}(0)$ having (at worst) an isolated singular
point at the origin then $\dim \cM (\cVnK ,f)\leq \mu (h)$
(notice that if $f=(g_1,\ldots ,g_p)$ defines an ICIS
then we can take a generic $\bC$-linear combination
$h=\sum _i a_ig_i$ having finite Milnor number $\mu (h)$).

(3) For $\VqG =\cVpA$ the moduli space $\cM (\cVpA ,f)$ is
finite dimensional for maps $f$ whose image (or discriminant)
has (at worst) an isolated singularity at the origin.
This applies to $\cA$-finite maps $f:(\bC ^n,0)\to (\bC ^p,0)$
with $p\ge 2n$ or $p=2$ (and any $n$). For the other
pairs of dimensions $(n,p)$ we only have the finiteness
results for $\cA$-s.q.h. maps (see (1) above).

(4) For $\VqG =\cVpA$ and $\cVnK$ we have the following
criterion for $\dim\cM (\VqG ,f)\ge 1$:
suppose $f_0$ is q.h. and the restriction of
$\gamma _{f_0}:LG\to LG\cdot f_0$ to the filtration-0 parts
of the modules in source and target has 1-dimensional kernel,
then the parameter $u$ of a deformation $f=f_0+u\cdot M$
by some non-zero element $M\in (NG\cdot f_0)_+$ is a
modulus for $\VqG$. Using this criterion in combination
with the existing $\cA$- and $\cK$-classifications in the
literature we conclude the following. Suppose
$f:(\bC ^n,0)\to (\bC ^p,0)$ is $\cA$-simple and
$n\ge p$ or $p=2n$ or $(n,p)=(2,3), (1,p)$ (and any corank)
or $(n,p)=(3,4)$ and corank 1 then: $f$ is w.q.h. if and
only if $\dim\cM (\cVpA ,f)=0$. Or suppose
that $f$ has $\cK$-modality at most one, $rank(df(0))=0$
and $n\ge p$ then: $f$ is q.h. if and only if
$\dim\cM (\cVnK ,f)=0$.

The contents of the remaining sections of this papers are
as follows.

\S 1. {\em Brief summary of earlier related works}: by
considering the moduli spaces $\cM (\VqG ,f)$ parameterizing the
$\VqG$-orbits inside $G\cdot f$ one can relate the seemingly
unrelated earlier works on volume-preserving diffeomorphisms
in singularity theory.

\S 2. {\em $H$-isotopic volume forms}: for a subgroup $H$ of the group
of diffeomorphisms Theorem \ref{2-sides} gives a criterion
for a  pair of volume forms to be $H$-isotopic, and Proposition
\ref{prEuler} gives a sufficient condition on $LH$ under which all
pairs of volume forms are $H$-isotopic. The results will be
applied to the subgroups $H=G_f^q$ defined above.

\S 3. {\em The moduli space $\cM (\VqG ,f)$}: the space
parameterizing the $\VqG$-orbits in a given $G$-orbit is
isomorphic to $C_r/\Div (LG_f^q)$ (Theorem \ref{thm-moduli})
and it vanishes for $\VqG$-w.q.h. maps $f$ (Proposition \ref{inf-wqh}).
These results imply, for example, that (over $\bC$) the stable
orbits for $\cVpA$ and $\cA$ and the simple orbits for
$\cVnK$ and $\cK$ agree (see Remark \ref{triv-applic}).

\S 4. {\em A cohomological description of $\cM (\VqG ,f)$ and some finiteness
results}: for finitely generated ideals $\cI$ in $C_n$ we define
a subcomplex $(\Lambda ^\ast (\cI ),d)$ of the de Rham complex
whose $n$th cohomology vanishes for w.q.h. ideals $\cI$ (Theorem
\ref{I-orb-wqh}). For $\cI =f^\ast\cM _p$ (not necessarily w.q.h.)
$H^n(\Lambda ^\ast (\cI ))$ is isomorphic to $\cM (\cVnK ,f)$
and is finite if $\cI$ contains the vanishing ideal of a
variety $W$ with (at worst) an isolated singular point at 0,
see Theorem \ref{finite1} (for a hypersurface germ $W$ we have
$H^n(\Lambda ^\ast (\cI ))\leq \mu (W)$, see Theorem \ref{finite2}).
These finiteness results imply for example: $\cM (\cVnK ,f)$
is finite if $f$ defines an ICIS, and $\cM (\cVpA ,f)$ is
finite for $p\ge 2n$ and $\cA$-finite $f$.

\S 5. {\em The foliation of $\cA$-orbits by $\cVpA$-orbits}:
in those dimensions $(n,p)$, for which the classification of
$\cA$-simple orbits is known, an $\cA$-simple germ $f$ is w.q.h. if
and only if $\cM (\cVpA ,f)=0$. The classifications of
the $\cVpA$-simple orbits in dimensions $(n,2)$ and $(n,2n)$,
$n\ge 2$, are described in Propositions
\ref{class(n,2)}, \ref{class(2,4)} and \ref{class(n,2n)}.
In \S\ref{sqh} the foliation of s.q.h. but not w.q.h. $\cA$-orbits
by $\cVpA$-orbits is investigated for $\cA$-unimodal germs
into the plane, and in \S\ref{multigerms} weak quasihomogeneity
is defined for multigerms under $\cVpA$-equivalence.

\S 6. {\em The foliation of $\cK$-orbits by $\cVnK$- and
$\cVpK$-orbits}: a $\cK$-unimodal germ $f$ of rank 0 is q.h.
if and only if $\cM (\cVnK ,f)=0$, and $\cM (\cVnK ,f)=0$ implies
$\cM (\cVpK ,f)=0$ (recall that germs $f$ of positive
rank are trivially w.q.h., hence their $\cK$-, $\cVnK$- and
$\cVpK$-orbits coincide).
Examples of rank 0 germs $f$ defining
an ICIS of codimension greater than one are presented for
which $\dim \cM (\cVnK ,f)<\mu (f)-\tau (f)$. For hypersurfaces
we have $\dim \cM (\cVnK ,f)=\mu (f)-\tau (f)$ (by a result
of Varchenko \cite{Va1}), in all our higher codimensional
examples we have $\dim \cM (\cVnK ,f)\leq \mu (f)-\tau (f)$
(and for s.q.h. germs $f$ it is easy to see that this inequality
holds in general).

\S 7. {\em The groups $\VqG \neq \cVpA$, $\cVnK$, $\cVpK$}:
in the final section we consider the remaining groups $\VqG$
for which there are $G$-finite singular maps (as opposed to
functions). Examples indicate that already
$G$-stable, singular and not trivially w.q.h. maps $f$ have
positive modality for these groups $\VqG$ (for $\cVnA$ the
fold map even has infinite modality).

\section{Brief summary of earlier related works}
\label{lit}

Having defined the moduli space $\cM (\VqG ,f)$ we can now
conveniently describe the known results within this framework.
Most of these results are on functions (hypersurface
singularities), and (as explained above) one can either fix $f$
and classify volume forms in the presence of a hypersurface
defined by $f$ (up to $G_f^q=\cR _f^n$, $\cA _f^n$ or $\cK
_f^n$-equivalence) or fix a volume form and classify functions up
to $\VqG =\cVnR$, $\cVnA$ or $\cVnK$-equivalence. Much less is
known for maps (see \S \ref{maps}).

\subsection{Results on functions (hypersurface singularities)}
\label{hypersurf}

First, consider $\cVnR$-equivalence for functions $f:(\bK
^n,\Omega _n,0)\to \bK$, $n\ge 2$. The isochore Morse-Lemma from
the late 1970s by Vey \cite{V} and Colin de Verdi\`ere and Vey
\cite{CdVV} gives a normal form for an $A_1$ singularity involving
a functional modulus. More recently isochore versal deformations
were studied in \cite{CdV} and \cite{Garay}. The following result
by Francoise \cite{JPF1,JPF2} generalizes the isochore
Morse-Lemma: let $b_1=1,b_2,\ldots ,b_{\mu (f)}$ be a base for
$N\cR _e\cdot f$ then
\[
\cM (\cVnR ,f)\cong \bK\{ (h_i\circ f)b_i :h_i\in C_1,
i=1,\ldots ,\mu (f) \}.
\]
Hence $f$ has precisely $\mu (f)$ functional moduli (the $h_i$ are
arbitrary smooth function-germs in one variable).

Second, for $\cVnA$ it is clear that (keeping the above
notation) $(h_1\circ f)1\in L\cL _e\cdot f$, hence
\[
\cM (\cVnA ,f)\cong \bK\{ (h_i\circ f)b_i :h_i\in C_1,
i=2,\ldots ,\mu (f) \}.
\]
This moduli space vanishes for an $A_1$ singularity,
and non-Morse functions $f$ have $\mu (f)-1$ functional moduli.

Finally, for $\cVnK$ the situation is much better.
The following generalization of the corresponding $\cK _f^n$
classification of volume forms has been studied, for
example, by Arnol'd \cite{AVolume}, Lando \cite{L1,L2},
Kostov and Lando \cite{K-L} and Varchenko \cite{Va1}:
given a hypersurface $f^{-1}(0)$ and a non-vanishing
function-germ $h$, classify $n$-forms
of the type $f^ahdx_1\wedge\ldots\wedge dx_n$ up to
diffeomorphisms that preserve $f^{-1}(0)$. For $a=0$
we have the special case of volume forms, and in this
case the result of Varchenko gives
\[
\cM (\cVnK ,f)\cong\langle f,\nabla f\rangle/\langle\nabla f\rangle ,
\]
which has dimension $\mu (f)-\tau (f)$. Both Francoise and
Varchenko made extensive use of results of Brieskorn \cite{B},
Sebastiani \cite{Se} and Malgrange \cite{Mal}
on the de Rham complex of differential
forms on a hypersurface with isolated singularities.

We will see that this dimension formula for $\cM (\cVnK ,f)$
does, in general, not hold for map-germs $f$ defining an ICIS
of codimension greater than one. The obvious counter-examples
are weakly quasihomogeneous maps $f$ that are not quasihomogeneous:
for such $f$ the dimension of the moduli space is zero,
but $\mu (f)-\tau (f)>0$. More subtle counter-examples
(Example \ref{FW} below) are the members of Wall's $\cK$-unimodal
series $FW_{1,i}$ of space-curves (which are not weakly
quasihomogeneous): here the dimensions of the moduli spaces
are equal to one and $\mu -\tau$ is equal to two.

\subsection{Results for maps}\label{maps}

Motivated by Arnold's classification of
$A_{2k}$ singularities of curves in a symplectic manifold
\cite{ASympl}
Ishikawa and Janeczko \cite{IJ} have (in our notation)
classified all $\cVpA$-simple map-germs
$f:(\bC ,0)\to (\bC ^2,\Omega _p,0)$. Notice that the
volume-preserving diffeomorphisms of $\bC ^2$ are also
symplectomorphisms. Looking at their classification
we observe that $\cM (\cVpA ,f)=0$ if $f$ is the germ
of a q.h. curve. Furthermore, it is shown in \cite{IJ}
that $\cod (\cA _{\Omega _p,e},f)=\delta (f)$, hence
the $\cA$-finiteness of $f$ (which is equivalent to $\delta (f)<\infty$)
implies the finiteness of the moduli space
$\cM (\cVpA ,f)$.

Notice that for $p=1$ any volume-preserving diffeomorphism
of $(\bK ^p,0)$ is the identity. For functions the groups
$\VqG$, where $q=n$, are therefore the only ones of interest,
and the results in \S \ref{hypersurf} (which could be
reproved using our approach) completely settle the classification
problem for function-germs in the volume-preserving case.
We will therefore concentrate on maps of target dimension
$p>1$ (but all general results also hold for $p=1$, of course).

\section{$H$-isotopic volume forms}
\label{iso}

In this section we study $H$-isotopies joining pairs of
volume forms for subgroups $H$ of $\cD _q:=\Diff (\bK ^q,0)$.
In the subsequent sections we will always apply these results to
the subgroups $H=G_f^q$ introduced in the introduction,
but it might be worth mentioning that the results of this
section have some additional applications, for example
to singularities of vector fields (and the proofs
remain valid for subgroups $H$ of the group of diffeomorphisms
of an oriented, compact, smooth $q$-dimensional manifold).

Let $\Lambda^k$ denote the space (of germs) of smooth
differential $k$-forms on $(\bK ^q,0)$, and denote
the subset of $\Lambda ^q$ of (germs of)  volume forms
by $\mbox{Vol}$. For a given subgroup $H\subset \cD _q$ we consider
a $C_q$-module $M$ in the Lie algebra $LH$ of $H$
(and $M=LH$ if $LH$ itself is a $C_q$-module). In the following
$\Omega$ and $\Omega _i$ always denote (germs of) volume forms
in $(\bK ^q,0)$.

\begin{defn}\label{H-diff}
We say that $\Omega _0$ and $\Omega _1$ are $H$-diffeomorphic if
there is a diffeomorphism $\Phi\in H$ such that
$\Phi^{\ast}\Omega_1=\Omega_0$
\end{defn}

\begin{defn}\label{IH-diff}
We say that $\Omega _0$ and $\Omega _1$ are $H$-isotopic if there
is a smooth family of diffeomorphisms $\Phi_t\in H$ for $t\in
[0,1]$ such that $\Phi_1^{\ast}\Omega_1=\Omega_0$ and
$\Phi_0=\mbox {Id}$.
\end{defn}

\begin{rem}
Two $H$-isotopic volume forms $\Omega _0$ and $\Omega _1$ are
obviously $H$-diffeomorphic. The converse is not true in
general. For example $dx_1\wedge dx_2$ and $-dx_1\wedge dx_2$ are
diffeomorphic but not isotopic, since any diffeomorphism mapping
one to the other changes orientation.
\end{rem}

\begin{defn}\label{M-eq}
We say that $\Omega _0$ and $\Omega _1$ are $M$-equivalent if
there is a vector field $X\in M$ such that $\Omega _0-\Omega
_1=d(X\rfloor \Omega)$ (for any volume form $\Omega$).
\end{defn}

\begin{rem}
Definition \ref{M-eq} does not depend on the choice of a volume
form $\Omega$. If $\Omega^{\prime}$ is another volume form then
$\Omega=f\Omega^{\prime}$ for some non-vanishing function $f$.
Then
$\Omega_1-\Omega_0=d(X\rfloor\Omega)=d(fX\rfloor\Omega^{\prime})$
and $fX\in M$ ($M$ being a module).
\end{rem}

\begin{thm}
\label{H-Moser} If $\Omega _0$ and $\Omega _1$ are $M$-equivalent
volume forms, which for $\mathbb K=\mathbb R$ define the same
orientation, then $\Omega _0$ and $\Omega _1$ are $H$-isotopic.
\end{thm}

\begin{proof} We use Moser's homotopy method \cite{Moser}. Let
$\Omega_t=\Omega_0+t(\Omega_1-\Omega_0)$ for $t\in [0,1]$. It is
easy to see that if $\Omega_0$ and $\Omega_1$ define the same
orientation then $\Omega_t\in \mbox{Vol}$ for any $t\in [0,1]$. We
are looking for a family of diffeomorphisms  $\Phi_t\in H$,
$t\in [0,1]$, such that
\begin{equation}\label{eq1}
\Phi_t^{\ast}\Omega_t=\Omega_0
\end{equation}
and $\Phi_0=\Id$. Differentiating (\ref{eq1}) we obtain
$$
\Phi_t^{\ast}(L_{Y_t}\Omega_t+\Omega_1-\Omega_0)=0,
$$
where $Y_t\circ\Phi_t=\frac{d}{dt}\Phi_t$, which implies that
\begin{equation}
d(Y_t \rfloor \Omega_t)=\Omega_0-\Omega_1 .
\end{equation}
But $\Omega_0$ and $\Omega_1$ are $M$-equivalent, hence there exists
a vector field $X\in M$ such that $\Omega _0-\Omega _1=d(X\rfloor
\Omega)$ for some volume form $\Omega$. We want to find a family
of vector fields $Y_t$ satisfying the following condition:
\begin{equation}
\label{d-homot} Y_t \rfloor \Omega_t=X\rfloor \Omega.
\end{equation}
But $\Omega_t=g_t \Omega$ for some non-vanishing smooth function
$g_t$. Hence $Y_t = (1/g_t)X$ is a solution of (\ref{d-homot}) and
$Y_t\in M$, because $X\in M$ and $M$ is a module. The vector field
$Y_t$ vanishes at the origin, hence its flow exists on some
neighborhood of the origin for all $t\in [0,1]$. Integrating $Y_t$
we obtain a smooth family of diffeomorphisms $\Phi_t\in H$ for
$t\in [0,1]$ such that $\Phi_0=\Id$ and
$\Phi_t^{\ast}\Omega_t=\Omega_0$, which implies that $\Omega_0$
and $\Omega_1$ are $H$-isotopic.
\end{proof}

Next, we will show that for subgroups $H$ of $\cD_q$ with $LH$
a submodule of the $C_q$-module $\theta_q$
the existence of an $H$-isotopy between a pair of volume forms
is equivalent to the $LH$-equivalence of this pair,
provided that $LH$ is closed with respect to integration in
the following sense.

\begin{defn}
We say $LH$  is closed with respect to integration if for any
smooth family $X_t\in LH$, $t\in[0,1]$, the integral
$\int_0^1X_tdt$ belongs to $LH$.
\end{defn}

\begin{thm}
\label{2-sides} Let $LH$ be a submodule of $\theta _q$,
which is closed with respect to integration.
Over $\mathbb K=\mathbb R$ we also assume that
$\Omega_0$ and $\Omega_1$ define the same orientation.
Then $\Omega _0$ and $\Omega _1$ are $LH$-equivalent
if and only if $\Omega_0$ and $\Omega_1$ are $H$-isotopic.
\end{thm}

\begin{proof}
The "only if" part follows directly from  Theorem \ref{H-Moser}.

For the converse, we require the following lemma

\begin{lem}
\label{pullback}
Let $\Phi_t$ be a smooth family of
diffeomorphisms and let $X_t$ be a family of vector fields such
that $\frac{d}{dt}\Phi_t=X_t\circ\Phi_t$. Then
$\frac{d}{dt}\Phi_t^{-1}=-(\Phi_t^{\ast}X_t)\circ \Phi_t^{-1}$.
\end{lem}

\begin{proof}[Proof of Lemma \ref{pullback}]
Differentiating $\Phi_t^{-1}\circ \Phi_t=\Id$ we obtain
$$0=\frac{d}{dt}(\Phi_t^{-1}\circ
\Phi_t)=\frac{d}{dt}(\Phi_t^{-1})\circ
\Phi_t+d(\Phi_t^{-1})\frac{d}{dt}\Phi_t,$$
which implies that
$\frac{d}{dt}(\Phi_t^{-1})=-d(\Phi_t^{-1})(X_t\circ\Phi_t)\circ\Phi_t^{-1}$.
But, by definition,
$\Phi_t^{\ast}X_t=d(\Phi_t^{-1})(X_t\circ\Phi_t)$.
\end{proof}

Returning to the proof of the theorem,
we assume that $\Omega_0$ and $\Omega_1$ are $H$-isotopic. Then
there exists, for all $t\in [0,1]$,
a smooth family of diffeomorphisms $\Phi_t\in H$ such that
$\Phi_0=\Id$ and $\Phi_1^{\ast}\Omega_0=\Omega_1$. Let
$(\Phi_t)'=\frac{d}{dt}\Phi_t=X_t\circ\Phi_t$, then
$$\Omega_1-\Omega_0=\Phi_1^{\ast}\Omega_0-\Omega_0
=\int_0^1(\Phi_t^{\ast}\Omega_0)'dt=
\int_0^1(\Phi_t^{\ast}\cL_{X_t}\Omega_0)dt
=\int_0^1\Phi_t^{\ast}d(X_t\rfloor\Omega_0)dt=$$
$$d\left(\int_0^1\Phi_t^{\ast}(X_t\rfloor\Omega_0)dt\right)
=d\left(\int_0^1(\Phi_t^{\ast}X_t)\rfloor\Phi_t
^{\ast}\Omega_0)dt\right)=
d\left(\int_0^1(\Phi_t^{\ast}X_t)\rfloor h_t \Omega_0)dt\right)
$$
for some smooth family of positive functions $h_t$.  Thus
$$\Omega_1-\Omega_0=
d\left(\int_0^1h_t\Phi_t^{\ast}X_tdt \rfloor\Omega_0\right).$$
Lemma \ref{pullback} implies $\Phi_t^{\ast}X_t\in LH$, and
using the fact that $LH$ is a
module we also have $h_t\Phi_t^{\ast}X_t\in LH$.
And $LH$ is closed with respect to integration, hence
$\int_0^1h_t\Phi_t^{\ast}X_tdt$ belongs to $LH$ too.
Therefore $\Omega_0$ and $\Omega_1$ are $LH$-equivalent,
as desired.
\end{proof}

\begin{defn}
The divergence of a vector field $X\in\theta _q$ with respect to a
given volume form $\Omega$ is, by definition, the smooth function
$\Div_{\Omega}(X)=d(X\rfloor\Omega)/\Omega$.
When the volume form $\Omega$ is understood from the context then
we simply write $\Div(X)$. And we have a map $\Div: \theta_q\to C_q$
defined by $X\mapsto \Div(X)$.
\end{defn}

\begin{cor}\label{mod}
Under the assumption of Theorem \ref{2-sides} the number of
$H$-moduli of volume forms is equal to
$$
\dim_{\bK}\frac{C_q}{\Div (LH)}.
$$
\end{cor}

\begin{proof} It is easy to see that spaces $C_q/\Div(LH)$ and
$\Lambda^q/\{d(X\rfloor\Omega):X\in LH\}$ are
isomorphic. By Theorem \ref{2-sides} the number of $H$-moduli of
volume forms is equal to the dimension of
$\mbox{Vol}/\sim _{LH}$. But it is easy
to see that the spaces $\Lambda^q/\{d(X\rfloor\Omega):X\in
LH\}$ and $\mbox{Vol}/\sim _{LH}$ are
equal if there exists a $X\in LH$ such that $d(X\rfloor\Omega)$ is a
volume form. Otherwise $\Lambda^q/\{d(X\rfloor\Omega):X\in
LH\}\setminus \mbox{Vol}/\sim _{LH}$ is
a linear subspace of positive codimension in
$\Lambda^q/\{d(X\rfloor\Omega):X\in LH\}$. This implies that
$$\dim_{\bK}
\frac{\Lambda^q}{\{d(X\rfloor\Omega):X\in LH\}}= \dim_{\bK}
\mbox{Vol}/\sim _{LH}.
$$
\end{proof}

Next, we describe two sufficient conditions for the existence of
a single $M$-equivalence class of volume forms in $(\mathbb K^q,0)$
(recall $M$ is a $C_q$-module in $LH$). For the first sufficient
condition we require the following

\begin{defn}\label{gen-Eul}
A linear vector field
$$
E_{w}=\sum_{i=1}^q w_i x_i\frac{\partial}{\partial x_i}.
$$
with integer coefficients $w_i$ is called a generalized
Euler vector field (for coordinates
$(x_1,\ldots,x_q)\in\bK ^q$ and weights $w=(w_1,\ldots,w_q)$).
\end{defn}

We first consider generalized Euler vector fields with
non-negative weights $w_i$ (for positive weights we obtain the
usual Euler vector fields). For $\cK_{\Omega_p}$-equivalence we
also require linear vector fields with negative coefficients (see
Theorem \ref{orb-wqh} below).

\begin{prop}\label{prEuler}
Let $X$ be the germ of a smooth vector field
on $(\mathbb K^q,0)$ which is locally diffeomorphic to a
generalized Euler vector field with non-negative weights
and positive total weight.
If $X$ generates a $C_q$-module in $LH$ then any two
germs of volume forms (which over $\bK =\bR$ define the same orientation)
are $H$-isotopic.
\end{prop}

\begin{proof} Let $E_{w}$ be (the germ of) the Euler
vector field for a coordinate system
$(x,y)=(x_1,\ldots,x_k,y_1,\ldots,y_{q-k})$ with weights
$w=(w_1,\ldots,w_k,0,\cdots,0)$, where $w_1,\cdots,w_k$ are
positive and let $\Omega_0$ be the germ of the volume-form
$dx_1\wedge\ldots\wedge dx_k \wedge dy_1\wedge\ldots\wedge
dy_{q-k}$. By Theorem \ref{H-Moser}, it is enough to show that for any
smooth $q$-form $\omega$ on $(\bK^q ,0)$
there exists a smooth function-germ $g$ on
$(\bK^q ,0)$ such that $\omega=d(g E_{w}\rfloor \Omega_0)$.

Let $G_t(x,y)=(e^{w_1
t}x_1,\ldots,e^{w_k t}x_k,y_1,\ldots,y_{q-k})$ for $t\le 0$. It is
easy to see that
$$
(G_t)':=\frac{d}{dt}G_t=E_{w}\circ G_t, \ G_0={\Id} , \
\lim _{t\rightarrow - \infty} G_t(x,y)=(0,y)
$$
for any $(x,y)\in\bK^q$. Thus
\begin{equation}
\label{integ} \omega=G_0^{\ast}\omega-\lim _{t\rightarrow - \infty}
G_t^{\ast}\omega=\int_{-\infty}^{0}(G_t^{\ast}\omega)'dt.
\end{equation}
But $\omega=f \Omega_0$ for some smooth
function-germ $f$ and
$$
(G_t^{\ast}\omega)'=G_t^{\ast}L_{E_{w}}\omega
=G_t^{\ast}d(E_{w}\rfloor\omega)=d(G_t^{\ast}(E_{w}\rfloor\omega)),
$$
hence
$$
(G_t^{\ast}\omega)'=d( G_t^{\ast}(E_{w}\rfloor f
\Omega_0))=d((f\circ G_t)G_t^{\ast}(E_{w}\rfloor \Omega_0)).
$$
One then checks by a direct calculation that
$G_t^{\ast}(E_{w}\rfloor
\Omega_0)=e^{t\sum_{i=1}^kw_i}(E_{w}\rfloor \Omega_0)$. Therefore
$(G_t^{\ast}\omega)'= d((f\circ
G_t)e^{t\sum_{i=1}^kw_i}(E_{w}\rfloor \Omega_0))$. Combining this
with (\ref{integ}) we obtain
$$
\omega=d(\int_{-\infty}^{0}((f\circ
G_t)e^{t\sum_{i=1}^kw_i})dt(E_{w}\rfloor
\Omega_0))=d(g(E_{w}\rfloor \Omega_0)),
$$
where $g$ is a function-germ on $(\bK^q ,0)$ defined as follows:
$$
g(x,y)=\int_{-\infty}^{0}(e^{t\sum_{i=1}^kw_i}(f(G_t(x,y)))dt.
$$
The function-germ $g$ is smooth, because
$$
\int_{-\infty}^{0}(e^{t\sum_{i=1}^kw_i}(f(G_t(x,y)))dt=\int_{0}^{1}(s^{\alpha}f(F_s(x,y))ds ,
$$
where $\alpha=(\sum_{i=1}^kw_i)-1$ and
$$
F_s(x_1,\ldots,x_k,y_1,\ldots,y_{q-k})=(s^{w_1}x_1,\ldots,s^{w_k}x_k,
y_1,\ldots,y_{q-k})
$$
for any $(x,y)=(x_1,\dots,x_k,y_1,\ldots,y_{q-k})$ and
$s\in[0,1]$. Multiplying the weights by a sufficiently large constant
we may assume that $\alpha>1$.
\end{proof}

We conclude this section by stating a second sufficient condition
for the existence of a single $M$-orbit of volume forms. Here we
assume that $LH$ contains a module $\mathcal M_q X$, where $X$ is
the germ of a non-vanishing vector field and $\mathcal M_q$ is the
maximal ideal of $C_q$.

\begin{prop}
\label{non-zero} If $X\in\theta _q$, $X(0)\neq 0$, and the
$C_q$-module $\mathcal M_q X$ is contained in $LH$ then any two
germs of volume forms (which over $\mathbb K=\mathbb R$ define the
same orientation) are $H$-isotopic.
\end{prop}

\begin{proof} $X(0)\ne 0$ implies that $X$ is diffeomorphic to
$\partial /\partial x_1$. Any germ of a $q$-form has in such a
coordinate system, for some $f\in C_q$, the following form
$$
f(x)dx_1\wedge dx_2 \wedge \cdots \wedge
dx_q=d(\int_0^{x_1}f(t,x_2,\cdots,x_q)dt \frac{\partial}{\partial
x_1}\rfloor dx_1\wedge dx_2 \wedge \cdots \wedge dx_q).
$$
And $\int_0^{x_1}f(t,x_2,\cdots,x_q)dt\partial /\partial x_1$
belongs to $\mathcal M_q\partial /\partial x_1$. Thus any two
germs of volume forms (which over $\bR$ define the same
orientation) are $H$-isotopic, by Theorem \ref{H-Moser}.
\end{proof}

\section{The moduli space $\cM (\VqG ,f)$}\label{moduli}

In this section we study smooth map-germs
$f:(\bK ^n,0)\to (\bK ^p,0)$ (for $\bK =\bC$ smooth means
complex-analytic, for $\bK =\bR$ smooth means either $C^{\infty}$
or real-analytic). We set $\cR :=\cD _n$ and
$\cL :=\cD _p$ (one can compose $f$ with elements of $\cD _n$ on
the right and with elements of $\cD _p$ on the left, which explains
this notation).

Let $G$ be one of the Mather groups $\cA ,\cK ,\cR ,\cL$
or $\cC$ (all of which can be considered as subgroups
of $\cA$ or $\cK$, e.g. $\cR\times 1\subset\cA$)
acting on the space of smooth map-germs
$f:(\bK ^n,0)\to (\bK ^p,0)$. And let $x=(x_1,\ldots ,x_n)$
and $y=(y_1,\ldots ,y_p)$ be coordinates on
$\bK ^n$ and $\bK ^p$, respectively.
The differential of the orbit map $g\mapsto g\cdot f$
($g\in G$ and the action on $f$ depends on the definition
of $G$)
\[
\g _f:LG\longrightarrow LG\cdot f
\]
has kernel $LG_f$ (where $G_f$ is the stabilizer of $f$ in $G$).
Recall that for $G=\cA$ the map $\g _f$ is given by
\[
L\cA =\cM _n\theta _n\oplus\cM _p\theta _p\to \cM _n\theta _f,
~~~(a,b)\mapsto tf(a)-{\om}f(b),
\]
where $tf(a)=df(a)$ and $wf(b)=b\circ f$,
and for $G=\cK$ it is given by
\[
L\cK =\cM _n\theta _n\oplus gl_p(C_n)\to \cM _n\theta _f,
~~~(a,B)\mapsto tf(a)-B\cdot f.
\]
The kernel of $\g _f$ inherits a $C_r$ module structure from
$LG$, where $r=p$ (or $r=n$) for $G$ a subgroup of $\cA$
(or $\cK$). Projecting onto source or target factors
\[
LG_f^n\longleftarrow LG_f\longrightarrow LG_f^p
\]
preserves this $C_r$ module structure. Denoting the factors
of $G_f$ by $G_f^n$ and $G_f^p$ their Lie algebras are the
above projections.
We also denote the factors of $G$ by $G^n$ and $G^p$
(hence e.g. for $G=\cA$ we have $G^n=\cR$).
Superscripts always denote projections onto one of the
factors.

Consider subgroups $\VnG$ and $\VpG$ of $G$
in which the diffeomorphisms (or families of diffeomorphisms
for $G=\cC$, see below)
preserve a given volume form $\Omega _n$ or $\Omega _p$
in the source or target, respectively.
For $r=n$ or $p$ and a given volume form $\Omega _r$ on $\bK ^r$
let $\Div :\cM _q\theta _q\to C_r$ be the
map that sends a vector field (vanishing at 0 in
$\bK ^n$ or $\bK ^p$) to its divergence.

For $\cK$-equivalence in combination with a volume form in
the target there are two ways to define the $\cVpC$ component.
But both version yield identical $\cVpK$-orbits (just as
the alternative definitions of $\cK$ yield the same $\cK$-orbits).

(1) In the original definition of $\cK$ by Mather,
$\cC$ consists
of diffeomorphisms $H=(\phi (x),\varphi (x,y))\in \cD _{n+p}$,
with $\varphi (x,0)=0$ for all $x\in (\bK ^n,0)$,
and the action on $f$ is given by
$H\cdot f:=\varphi (x,f\circ\phi (x))$.
We can think of $H$ as a $n$-parameter family of diffeomorphisms
$\{ \varphi _x\}$, $x\in\bK ^n$, acting on $f$ by sending
$x$ to $\varphi _x\circ f(x)$.
If $\Omega _p$ is a volume
form on $(\bK ^p,0)$ we require that each
$\varphi _x$ preserves $\Omega _p$
(i.e. $\varphi _x^{\ast}\Omega _p=\Omega _p$
for all $x\in (\bK ^n,0)$). In this way we obtain
a subgroup $\cVpC$ of $\cC$, and $\cVpK :=\cR\cdot \cVpC$
(semi-direct product).

(2) In the linearized version of $\cK$ we set
$\cC :=GL_p(C_n)$ and restrict
to $\cVpC =SL_p(C_n)$, then $L\cVpC =sl_p(C_n)$ consists
of $p\times p$ matrices over $C_n$ with zero trace.
And, again, $\cVpK :=\cR\cdot \cVpC$.
Then $\Div$ can be considered as a map $B\mapsto\trace B$
as follows:
the map $gl_p(C_n)\to \cM _n\theta _f$, sending
$B$ to $B\cdot f$ (multiplication of $f$ as a column vector
of its component functions by a matrix $B=(b_{ij})$), can
also be written $B\cdot f=X_B\circ f$, where
$X_B=\sum _{i=1}^p (b_{i1}(x)y_1+\ldots +b_{ip}(x)y_p)\partial /\partial y_i$
is a linear vector field in $\bK ^p$ with coefficients
$b_{ij}\in C_n$. Hence $\Div X_B =\trace B\in C_n$.

For any of the above volume preserving subgroups $\VqG$
of $G$ we have the following

\begin{prop}\label{Omega-moduli}
For $q=n$ or $p$, and $\Div :\cM _q\theta _q\to C_r$
(where $r=n$ for $G_f^q=\cK _f^p$ and $r=q$ in all other cases),
we have an isomorphism
\[
\cM (\VqG ,f):= \frac{LG\cdot f}{L\VqG\cdot f}\cong
\frac{C_r}{\Div (LG_f^q)}.
\]
\end{prop}

\begin{proof}
Let $\pi :LG\to LG^q$ be the projection onto one of the factors,
so that for $u=(a,b)$ we have $v:=\pi (u)$ is equal to
$a\in\cM _n\theta _n$ or $b$, where either $b\in\cM _p\theta _p$
(for $G=\cA$) or $b=X_B$ for some $B\in gl_p(C_n)$ (for $G=\cK$).
(Recall that in the latter case $\Div (X_B)=\trace B$.)
Then consider the epimorphism
\[
\be : LG\To C_r,~~~~~u\mapsto \Div (v).
\]
Factoring out the kernel we obtain an isomorphism
\[
\bar\be :\frac{LG}{L\VqG}\To C_r.
\]
We also have a well-defined map
\[
\g :\frac{LG}{L\VqG}\To
\frac{\cM _n\cdot\theta _f}{L\VqG\cdot f}
\]
sending $[(a,b)]$ to $[tf(a)-{\om}f(b)]$ (for $G$ a subgroup
of $\cA$) and $[(a,B)]$ to $[tf(a)-B\cdot f]$ or, equivalently,
$[(a,X_B)]$ to $[tf(a)-X_B\circ f]$
(for $G$ a subgroup of $\cK$). We see that
\[
\im\g =\frac{LG\cdot f}{L\VqG\cdot f}
\]
and that $\bar\be (\ker\g )=\Div (LG_f^q)$.
Factoring out the kernel of $\g$ yields an
isomorphism $\bar\g$ onto $\im\g$
so that $\bar\be\circ\bar\g ^{-1}$ is the desired
isomorphism.
\end{proof}

\begin{rem}
For $G=\cA$ the vector fields
$(a,b)\in L\cA _f$, $b\in L\cA _f^p$ and $a\in L\cA _f^n$
are also said to be $f$-related, liftable and lowerable,
respectively.
\end{rem}

Notice that $LG_f^q$ inherits a $C_r$ module structure,
where $r=n$ or $p$, from $LG_f$ and $LG$. In fact, we have

\begin{lem}\label{submodule}
$LG_f$ is a $C_r$-submodule of $LG$ ($r=p$ or $n$ for $G$
a subgroup of $G=\cA$ or $\cK$, respectively), which is
closed under integration.
The same is true for the factors $LG_f^q$ of $LG_f$.
\end{lem}

\begin{proof}
The statements about the module structure are obvious.
And for 1-parameter families of
vector fields $v_t=(a_t,b_t)$ (for $G=\cA$) or $(a_t,X_{B_t})$
(for $G=\cK$), $t\in [0,1]$, in the kernel of $\g _f$ we have
$0=\int _0^1 \g _f(v_t)dt=\g _f (\int _0^1 v_t dt)$,
hence $\int _0^1 v_t dt\in LG_f$.
And it is clear that the $q$-component of $\int _0^1 v_t dt$
belongs to $LG_f^q$.
\end{proof}

We can now deduce from Proposition \ref{Omega-moduli}
and Corollary \ref{mod} the following

\begin{thm}\label{thm-moduli}
For all volume preserving subgroups $\VqG$ of $G$,
except for $\cVnA$, the dimension of
\[
\cM (\VqG ,f):= \frac{LG\cdot f}{L\VqG\cdot f}\cong
\frac{C_r}{\Div (LG_f^q)}
\]
is equal to the number of $\VqG$-moduli of $f$ and also
to the number of $G_f^q$-moduli of volume forms in
$(\bK ^q,0)$. (For $\cVnA$ the above statement holds
in the formal category, in the smooth category the
number of moduli is at least $\dim\cM (\cVnA ,f)$.)
\end{thm}

\begin{proof}
In all cases, except $L\cA _f^n$, the component $LG_f^q$ of $LG_f$
is a module over the ring $C_r$ appearing as the target of the map
$\Div :\cM _q\theta _q\to C_r$. And $LG_f^q$ is closed under
integration, by the above lemma, hence Corollary \ref{mod}
applies. For $L\cA _f^n$ we notice that Proposition
\ref{Omega-moduli} is a statement about vector spaces (a $C_r$
module structure is not required).
\end{proof}

\begin{remark}\label{modality}
At this point it is perhaps useful to briefly recall the following.
The $G$-modality of a map-germ $f$ is, roughly speaking, the
least $m$ such that a small neighborhood of $f$ can be covered
by a finite number of $m$-parameter families of $G$-orbits.
(More precisely, we consider the $j^k(G)$-orbits in some neighborhood
of $j^kf$ in a finite-dimensional jet-space $J^k(n,p)$
for some $k$ for which all these $j^k(G)$-orbits are $G$-sufficient
-- recall that the $G$-determinacy degree of $f$ in general fails
to be upper semicontinuous under deformations of $f$, see
\cite{Wa} for a survey of results on $G$-determinacy.)
Map-germs $f$ of $G$-modality $0,1,2,\ldots$ are said to be
$G$-simple, $G$-unimodal, $G$-bimodal and so on. An
$m$-$G$-modal family depends on no more than $m$ parameters (moduli),
for $G=\cR$ and function-germs it depends on exactly $m$ moduli
\cite{Gabr}. For a subgroup $\VqG$ of a Mather group $G$
and an $m$-parameter family of map-germs $f^{\lambda}$
the dimension of $\cM (\VqG ,f^{\lambda})$ is equal to the
number of $\VqG$-moduli of $f^{\lambda}$, and also
to the number of $G_{f^{\lambda}}^q$-moduli of volume forms in
$(\bK ^q,0)$, for each fixed vector $\lambda\in\bK ^m$ of
$G$-moduli of $f^{\lambda}$.
\end{remark}

We are now interested in classes of map-germs $f$
for which the moduli spaces $\cM (\VqG ,f)$
vanish. For the groups $\VqG =\cVpA$, $\cVnK$
and $\cVpK$
such classes of maps are given by the following
weak forms of quasihomogeneity.

\begin{defn}\label{def-wqh}
A map-germ $f:(\bK ^n,0)\to (\bK ^p,0)$,
which is q.h. for weights $w_i\in\bZ$ ($1\le i\le n$)
and weighted degrees $\de _j$ ($1\le j\le p$), is said to
be {\em weakly quasihomogeneous} (w.q.h.)
for the group $\VqG$ if the following
conditions hold.
\begin{itemize}
\item For $\VqG =\cVpA$: all $\de _j\ge 0$ and $\sum _j \de _j>0$.
\item For $\VqG =\cVnK$: all $w _i\ge 0$ and $\sum _i w_i>0$.
\item For $\VqG =\cVpK$: $\sum _j \de _j\neq 0$.
\end{itemize}
\end{defn}

\begin{rem}
(i) The condition w.q.h. depends on the group $\VqG$, when the
group is clear from the context we will simply say that $f$ is
w.q.h.

(ii) For any subgroup $\VqG\neq \cVnA$ of $G$ we have the following
{\em ``trivial versions of w.q.h''} for $f$:
(1) for $q=p$ and $f$
$G$-equivalent to some map-germ having a zero
component function, and (2) for $q=n$ and $df(0)$ of positive
rank. For $\VqG =\cVpA$, $\cVnK$ and $\cVpK$ it is easy to see
that ``trivially w.q.h.'' is a special case of w.q.h.:
for (1) we give the zero component function positive weighted
degree (and set all weights $w_i$ or all other degrees $\de _j$
to zero), and for (2) we have (up to $G$-equivalence)
$f=(x_1,g(x_2,\ldots ,x_n))$, so we take $w_1=1$ and $w_i=0$, $i>1$.
\end{rem}

We then have the following

\begin{prop}\label{inf-wqh}
Let $f$ be w.q.h. for one of the groups
$\VqG =\cVpA$, $\cVnK$ or $\cVpK$ (or ``trivially w.q.h.''
for any group). Then $\cM(\VqG ,f)=0$.
\end{prop}

\begin{proof}
We will show that
$LG ^q\cdot f\subset L\VqG\cdot f$
(here $LG^q\cdot f$ is one of the factors of $LG\cdot f$),
so that $L\VqG\cdot f = LG\cdot f$.

For $\VqG =\cVpA$ we have to show that
$L\cL\cdot f\subset L\cVpA\cdot f$. Clearly it
is enough to check this inclusion for the elements
of $L\cL\cdot f$ that do not belong to $L\cVpL\cdot f$.
Let $y^{\al}=\prod _l y_l^{\al _l}$ and $| \al | \ge 0$.
The following elements of $L\cVpA\cdot f$ yield
${\om}f(y^{\al}y_i\cdot
\partial /\partial y_i)\in L\cL\cdot f$, $i=1,\ldots, p$:
\[
{\om}f(-(1+\al _j)y_1y^{\al }\cdot \partial /\partial y_1 +(1+\al
_1)y_jy^{\al}\cdot \partial /\partial y_j),~~~j=2,\ldots ,p
\]
together with
\[
tf(f^{\ast}(y^{\al})\sum _{i=1}^n w_ix_i\cdot \partial /\partial
x_i) -\sum _{j=2}^p \de _j\cdot {\om}f
\Big( -\frac{1+\al _j}{1+\al
_1}y^{\al}y_1 \cdot \partial /\partial y_1 +y^{\al}y_j\cdot
\partial /\partial y_j\Big)
\]
\[
=(1+\al _1)^{-1}\sum _{j=1}^p (1+\al _j)\de_j \cdot
{\om}f(y^{\al}y_1\cdot
\partial /\partial y_1) .
\]
Notice that $\sum _j (1+\al _j)\de _j\ne 0$, for any exponent
vector $\al $, is equivalent to $f$ being w.q.h. for the
group $\cVpA$.

For $\VqG =\cVnK$ we have to show that
$L\cR\cdot f\subset L\cVnK\cdot f$. Exchanging the roles
of the source and target vector fields, we see that
the following elements of $L\cVnK\cdot f$ yield
$tf(x^{\al}x_i\cdot
\partial /\partial x_i)\in L\cR\cdot f$, $i=1,\ldots, n$:
\[
tf(-(1+\al _j)x_1x^{\al }\cdot \partial /\partial x_1 +(1+\al
_1)x_jx^{\al}\cdot \partial /\partial x_j),~~~j=2,\ldots ,n
\]
together with
\[
x^{\al}\sum _{i=1}^n\de _i f_i\cdot\partial /\partial y_i
-\sum _{j=2}^n tf\Big( w_j \Big(
-\frac{1+\al _j}{1+\al _1}x_1x^{\al} \cdot \partial /\partial x_1
+x_j x^{\al}\cdot \partial /\partial x_j\Big)\Big)
\]
\[
=(1+\al _1)^{-1}\sum _{j=1}^n (1+\al _j)w_j \cdot
tf(x_1x^{\al}\cdot\partial /\partial x_1) .
\]
Notice that $\sum _j (1+\al _j)w_j\ne 0$, for any exponent
vector $\al $, is equivalent to $f$ being w.q.h. for the
group $\cVnK$.

For $\VqG =\cVpK$ we have to show that
$L\cC\cdot f\subset L\cVpK\cdot f$. Notice that
$L\cVpC =sl_p(C_n)$ consists of elements $B$ of $gl_p(C_n)$
with trace 0, hence we have a $C_n$-module structure.
Therefore, if $E_{ij}$ denotes a $p\times p$ matrix
with entry $(i,j)$ equal to 1 and all other entries 0
then it is enough to show that
$E_{ii}\cdot f\in L\cVpK\cdot f$, for $i=1,\ldots ,p$.
(Notice that this implies that $L\cC\cdot f\subset L\cVpK\cdot f$,
both for the linearized version $GL_p(C_n)$ of $\cC$
and for Mather's original $\cC$, because of the $C_n$-module
structure.)
Taking for $j=2,\ldots ,p$
\[
-f_1\cdot\partial /\partial y_1+f_j\cdot\partial /\partial y_j
\]
(corresponding to $(E_{jj}-E_{11})\cdot f$ with
$(E_{jj}-E_{11})\in L\cVpC$) and
\[
tf\big( \sum _{i=1}^n w_ix_i\cdot\partial /\partial x_i\big)
-\big( (-\de _2-\ldots -\de _p)E_{11}+\de _2E_{22}+\ldots
+\de _pE_{pp}\big)\cdot f
\]
\[
= \sum _{j=1}^p\de _jf_1\cdot\partial /\partial y_1
\]
we see that
$E_{ii}\cdot f\in L\cVpK\cdot f$ ($i=1,\ldots ,p$)
provided that $\sum _j \de _j\neq 0$.

Finally, by the remark in the introduction, there is nothing
to prove in the ``trivially w.q.h. cases'' (arbitrary
diffeomorphisms in a proper subspace can be extended to volume
preserving diffeomorphisms of the total space $(\bK ^q,0)$).
\end{proof}

The proposition says that, at the infinitesimal level,
the tangent spaces of the $G$-orbit and of the
$\VqG$-orbit of $f$ coincide.
For $\bK =\bR$ let $G^+$ be the
subgroup of $G$ for which the diffeomorphisms of
the $G^q$ factor of $G$ are orientation preserving.
We then have at the level of orbits the following

\begin{thm}\label{orb-wqh}
Let $f:(\bK ^n,0)\to (\bK ^p,0)$ be w.q.h. for one of
the groups
$\VqG =\cVpA$, $\cVnK$ or $\cVpK$ (or ``trivially w.q.h.''
for any group) then:

(i) any two volume forms $\Omega$, $\Omega '$
on $\bK ^q$ (so that, in the case of $\bK =\bR$,
$\Omega |_0$ and $\Omega '|_0$ define the same
orientation in $T_0\bR ^q$)
are $G_f^q$-isotopic.

(ii) $f'\sim _G f$ (for $\bK =\bC$) and $f'\sim _{G^+} f$ (for
$\bK =\bR$) imply $f'\sim _{\VqG} f$ (for some given volume form
$\Omega _q$ on $\bK ^q$).
\end{thm}

\begin{proof}
Using the weights $w_i$ (for $q=n$) or weighted degrees
$\de _j$ (for $q=p$) in the definition of a
$\VqG$-w.q.h. map $f$ we can define generalized Euler vector
fields in $\bC ^q$. For $\VqG =\cVpA$ and $\cVnK$ the vector
fields have non-negative coefficients, hence Proposition \ref{prEuler}
implies statement (i).  For $\cVpK$ we can have negative coefficients
and we deduce statement (i) by a slightly modified argument (see below).
The equivalence of (i) and (ii) is clear (over $\bC$ the $G$-orbits
are connected).

For $\cK_{\Omega_p}$-equivalence the weighted degrees $\delta _i$
of $f$ yield a generalized Euler vector field $E_\delta
=\sum_{i=1}^p \delta_i y_i \partial /\partial y_i$ in $(\mathbb
K^p ,0)$.  We first claim that any volume form $\Omega_p$ is
$\cK^p_f$-equivalent to some linear volume form $g(x)dy_1\wedge
\cdots dy_p$ parameterized by $g\in C_n$ with $g(0)\neq 0$. Let
$\Psi$ be an origin-preserving diffeomorphism of $(\mathbb K^p,0)$
such that, for $\Omega _p=h(y)dy_1\wedge\cdots \wedge dy_p$, we
have $\Psi^{\ast}\Omega_p=dy_1\wedge\cdots \wedge dy_p$. Its
inverse has the form
$$
\Psi^{-1}(y)=(\sum_{i=1}^p\phi_{1i}(y)y_i,\cdots,
\sum_{i=1}^p\phi_{pi}(y)y_i).
$$
We have $\Psi ^{-1}\circ f(x)=\Phi _x\circ f(x)$ for the following
family $\Phi _x$ of diffeomorphisms of $(\mathbb K^p,0)$
parameterized by $x\in (\mathbb K^n,0)$
$$
\Phi_x(y)=(\sum_{i=1}^p\phi_{1i}(f(x))y_i,\cdots,
\sum_{i=1}^p\phi_{pi}(f(x))y_i).
$$
 Hence $\Psi\circ\Phi
_x\circ f=f$ (i.e., $\Psi\circ\Phi _x\in\cK _f^p$) and
$\Phi_x^{\ast}\Psi^{\ast}\Omega_p=g(x)dy_1\wedge \cdots \wedge
dy_p$, where $g(x)=\det (d\Phi _x )$. Clearly $g(0)\neq 0$, which
implies the above claim.

It is therefore sufficient to consider the equivalence of
parameterized linear volume forms. Notice that
$E_\delta$ generates a $C_n$-submodule of $L\cK _f^p$ and
$$
g(x)dy_1\wedge \cdots \wedge dy_p= d\left(\frac{g(x)}{\sum_{i=1}^p
\delta_i}E_{\delta}\rfloor dy_1\wedge \cdots \wedge dy_p\right)
$$
(recall that $\sum_{i=1}^p \delta_i\ne 0$), hence any pair of
such volume forms is $L\cK _f^p$-equivalent. Furthermore, by the
argument in the proof of Theorem \ref{H-Moser}, such a pair of
volume forms (which, for $\mathbb K=\mathbb R$, is required to
define the same orientation) is $\cK^p_f$-isotopic.
\end{proof}

``Non-trivial applications'' of the above result -- namely to
weakly quasihomogeneous map-germs $f$ that are neither
quasihomogeneous nor trivially weakly quasihomogeneous --
will be considered later.  For quasihomogeneous and
trivially quasihomogeneous germs $f$ we have the following
immediate applications.

\begin{rem}\label{triv-applic}
(1) Quasihomogeneous case: all $\cA$-stable and all $\cK$-simple
map-germs $f$ are quasihomogeneous. Hence the classifications,
over $\bC$, of stable germs for $\cA$ and $\cVpA$ and of simple germs for
$\cK$, $\cVnK$ and $\cVpK$ agree -- over $\bR$, each $\cA$-stable
or $\cK$-simple orbit corresponds to one or two stable or simple orbits
for the volume preserving subgroups.

(2) Trivially weakly quasihomogeneous case: (i) the classifications
of map-germs $f$, with $df(0)$ of positive rank, for the groups
$\cK$, $\cVnK$ and $\cVpK$ agree. (ii) For map-germs
$f:(\bK ^n,0)\to (\bK ^p,0)$ whose image lies in a proper
submanifold of $(\bK ^p,0)$ (such $f$ have, up to a target coordinate
change, a zero component function) the $\cA$- and $\cVpA$-orbits,
the $\cL$- and $\cVpL$-orbits, and the $\cC$- and $\cVpC$-orbits
agree. Notice, for example, that the $\cA$ and $\cVpA$ classifications
of simple curve-germs agree for $p\ge 7$ (Arnol'd
\cite{ACurves} has shown
that all stably simple curves can be realized in 6-space,
hence all $\cA$-simple curves in higher dimensions have
zero component functions).
\end{rem}

\section{A cohomological description of $\cM (\VqG ,f)$
and some finiteness results}\label{cohom}

The results on $\cM (\cVnK ,f)$ can be reformulated
for ideals, and for this reformulation we obtain a
further isomorphism in terms of cohomology.
This cohomological description yields some finiteness
results in the non-w.q.h. case.
Let $\cI\subset C_n$ be a finitely generated ideal
(recall: for $C_n=\cO _n$ all ideals are f.g., for
$C_n=\cE _n$, the ring of $C^{\infty}$ function germs,
there are non-f.g. ideals like $\cM _n^{\infty}$).

We say that $\cI$ and $\cJ$
are $\cD _n$-equivalent if and only if there is
a diffeomorphism germ $\phi\in\cD _n$ such that
$\phi ^{\ast}\cI =\cJ$. The stabilizer of $\cI$
is $(\cD _n)_{\cI}=\{ \phi :\phi ^{\ast}\cI =\cI \}$,
and
\[
L(\cD _n)_{\cI}=\Derlog (\cI )=\{ Y\in\theta _n : Y\cI \subset
\cI\} ,
\]
where, for $h\in\cI$, we set $Yh :=dh\cdot Y$.
For
$\cI =\langle g_1,\ldots ,g_p\rangle$ and
$f:=(g_1,\ldots ,g_p):(\bK ^n,0)\to (\bK ^p,0)$
we have the following:
\[
\phi^{\ast}\langle g_1,\ldots ,g_p\rangle
:=\langle g_1\circ\phi ,\ldots ,g_p\circ\phi\rangle =
\langle g_1,\ldots ,g_p\rangle
\]
if and only if $f=B\cdot (f\circ\phi )$ for some
$B\in {\rm GL}_p(C_n)$. Hence
$\Derlog (\cI )=L\cK _f^n$,
and setting
$\cD _{\Omega _n} :=\{ h\in\cD _n: h^{\ast}\Omega _n=\Omega _n\}$
(for a given volume form $\Omega _n$ in $(\bK ^n,0)$) we
have the following isomorphisms for the (infinitesimal)
$\cD _{\Omega _n}$-moduli space of $\cI$:
\[
\cM (\cD _{\Omega _n},\cI):=\frac{C_n}{\Div (\Derlog(\cI ))}\cong
\frac{C_n}{\Div (L\cK _f^n)}\cong
\frac{L\cK\cdot f}{L\cVnK\cdot f}.
\]
This moduli space is also isomorphic to the $n$th cohomology group
of the following complex $(\Lambda^\ast(\cI),d)$.
Defining
for $k=0,\ldots ,n$ the vector spaces
$$
\Lambda^k(\cI):= \{\alpha+d\beta\in \Lambda^k: d\cI\wedge \alpha
\subset \cI \Lambda^{k+1}, \ d\cI\wedge \beta \subset \cI
\Lambda^{k}\}
$$
we obtain a subcomplex $(\Lambda^\ast(\cI),d)$
of the de Rham complex $(\Lambda^{\ast},d)$.
Sometimes we shall simply write
$\Lambda^\ast(\cI)=(\Lambda^\ast(\cI),d)$ and similarly for
the other complexes defined below (the differential is always
the same $d$).

The $n$th cohomology group of the complex $(\Lambda^\ast(\cI),d)$ is
$$
H^n((\Lambda ^\ast (\cI),d)=\Lambda^n /
d\Lambda^{n-1}(\cI)=\Lambda^n / \{d\alpha \in \Lambda^n:
d\cI\wedge \alpha \subset \cI \Lambda^n\}.
$$
For a given volume form $\Omega_n$ the map
$$\Derlog(\cI)\ni X\mapsto
X\rfloor \Omega_n \in \{\alpha\in \Lambda^{n-1}:d\cI\wedge \alpha
\subset \cI \Lambda^n\}$$
is an isomorphism. Notice that the tangent space to
$\Lambda ^n$ can be identified with $C_n$, and recall that
$\Div_{\Omega _n}(X)=d(X\rfloor\Omega _n)/\Omega _n$.
Hence we see that
\[
H^n((\Lambda ^\ast (\cI),d))\cong C_n/\Div (\Derlog(\cI ))
\cong \cM (\cD _{\Omega _n},\cI).
\]
Furthermore, Theorem \ref{2-sides} implies the
following

\begin{prop}
Two volume forms (defining the same orientation) are $(\cD
_n)_{\cI}$-isotopic if and only if they define the same cohomology
class in $H^n((\Lambda^{\ast}(\cI),d))$.
\end{prop}

\begin{defn}We say that an ideal $\cI$ in $C_n$ is w.q.h. if it has a set
of generators $g_1,\ldots ,g_p$ such that the corresponding map
$f=(g_1,\ldots ,g_p)$ is $\cVnK$-w.q.h. (notice that this is a
natural generalization of homogeneous ideals).
\end{defn}
\begin{rem}
If the ideal $\cI$ is w.q.h. then the variety defined by $\cI$ is
``quasihomogeneous with respect to a smooth submanifold'' in the
sense of \cite{DJZ2}.
\end{rem}

We can now reformulate Theorem \ref{orb-wqh} as follows

\begin{thm}\label{I-orb-wqh}
Let $\cI$ be a w.q.h. ideal in $C_n=\cO _n$ or $\cE _n$.
For $C_n=\cE _n$ we assume that $\cI$ is finitely generated,
and (over $\bR$)  $\cD ^+_n$ denotes the group of
orientation preserving diffeomorphisms.
Then we have the following:

(i) any two volume forms on $\bK ^n$ (which, in the
case $\bK =\bR$, define the
same orientation in $T_0\bR ^n$)
can be joined (via pullback) by
a 1-parameter family of diffeomorphisms $\phi _t$
such that $\phi _t ^{\ast}\cI =\cI$ (i.e., by a
$(\cD _n)_{\cI}$-isotopy).

(ii) For a given volume form $\Omega _n$, let $\cD _{\Omega _n}$
be the subgroup of $\cD _n$ whose elements preserve $\Omega _n$.
Then
$\phi ^{\ast}\cI =\cJ$ for some $\phi\in\cD _n$
(for $\bK =\bC$) or some $\phi\in\cD ^+ _n$ (for $\bK =\bR$)
implies $h^{\ast}\cI =\cJ$ for some $h\in\cD _{\Omega _n}$.
\end{thm}

\begin{rem}
For $\cA$-equivalence we have the following cohomological
description. Given a map-germ $f:(\bC ^n,0)\to (\bC ^p,0)$,
let $\Delta _f$ be the discriminant (for $n\ge p$) or the image
(for $n<p$) of $f$. If $f$ satisfies the necessary and sufficient
condition (namely, GTQ for $n\ge p$ or NHS for $n<p$)
for the equality $\Derlog (\cI (\Delta _f)) ={\rm Lift}(f)$ of
Theorem 2 in \cite{BdPW} then we have the following isomorphism:
\[
\cM (\cVpA ,f)\cong H^p((\Lambda ^\ast (\cI (\Delta _f),d)),
\]
here $\cI (\Delta _f)$ is the vanishing ideal of
$\Delta _f\subset (\bC ^p,0)$. For the precise definitions
of GTQ (generically a trivial unfolding of a q.h. germ)
and NHS (no hidden singularities) we refer to \cite{BdPW}.
\end{rem}

Notice that $f$ w.q.h. (for $\cVpA$) implies that
the ideal $\cI (\Delta _f)$ is weakly quasihomogeneous.
But there are w.q.h. maps $f$ that fail to be GTQ.
We give two examples illustrating these facts.

\begin{ex}\label{neg-w}
The map $f:(\bC ^3,0)\to (\bC ^2,0)$ given by
$f(x,y,x)=(x,xy+y^5+y^7z)$ is w.q.h. with weights
$(4,1,-2)$ and weighted degrees $(4,5)$. The discriminant
of $f$ is the origin in $(\bC ^2,0)$. The critical set
is the $z$-axis, which consists of $\cA$-unstable points,
hence $f$ fails to be $\cA$-finite.
\end{ex}

\begin{ex}\label{eBdPW}
In \cite{BdPW}
$$f:(\bC^3,0)\rightarrow (\bC^2,0),~~ f(u,x,y)=(u,x^4+y^4+ux^2y^2)$$
is presented as an example of a non-GTQ map-germ.
But $f$ is
weakly quasi-homogeneous (for the weights $(0,1,1)$).
Notice that, again, $f$ fails to be $\cA$-finite.
\end{ex}

One may not care much about such degenerate examples of infinite
$\cA$-codimension.
In Section \ref{fol-A} we describe more subtle examples of weakly
quasi-homogeneous map-germs that are $\cA$-finite
and even $\cA$-simple.

Next, we will derive some finiteness results for
$H^n(\Lambda^{\ast}(\cI))$ when $\cI$ is not necessarily w.q.h.,
and apply these to deduce $\VqG$-finiteness from $G$-finiteness of
$f:(\bC ^n,0)\to (\bC ^p,0)$ (for certain $G$ and $(n,p)$). We
assume here that $\mathbb K=\mathbb C$ and that all germs (at $0$)
are $\mathbb C$-analytic. For $\cI =\langle g_1,\ldots ,g_s
\rangle$ we denote the ideal of maximal minors of the Jacobian of
$g=(g_1,\ldots ,g_s)$ (viewed as a map-germ) by $J(g)$, and we set
$\nabla g_i:=J(g_i)$. Recall that $\langle g_1,\ldots
,g_s,J(g)\rangle$ is the vanishing ideal of the set of
$\cK$-unstable points of $g$, so that (by the Nullstellensatz) $g$
is $\cK$-finite if and only if  $\cM _n^r\subset \langle
g_1,\ldots ,g_s,J(g)\rangle$, for some $r<\infty$, or iff $g$ has
(at worst) an isolated singular point at $0$. Also notice that
\[
\langle g_1,\ldots ,g_s,J(g)\rangle\subset
\langle g_1,\ldots ,g_s,\nabla g_1,\cdots,\nabla g_s\rangle
\]
implies that, for $\cK$-finite $g$, the ideal on the RHS of
this inclusion has finite colength.

We will relate the complex $\Lambda^{\ast}(\cI)$ to the following
subcomplex of the de Rham complex: $(\cA^{\ast}_0(\cI),d)$, where
$\cA^k_0(\cI)=\{\alpha+d\beta\in \Lambda^k:\alpha\in \cI\Lambda^k,
\ \beta\in \cI\Lambda^{k-1}\}$. If $\cI$ is the vanishing ideal of
a variety $V$ then this complex is called the complex of zero
algebraic restrictions to $V$ (see \cite{DJZ1}, \cite{DJZ2},
\cite{D}). The cohomology of the quotient complex
$(\Lambda^{\ast}/\cA^{\ast}_0(\cI(V)),d)$ has been studied in
detail in earlier works (see \cite{Re},\cite{B-H},\cite{B},
\cite{Se},\cite{Gr1},\cite{Gr2}). Notice that the  $k$-th
cohomology $H^k(\Lambda^{\ast}/\cA^{\ast}_0(\cI))$ of this
quotient complex and the $(k+1)$-th cohomology
$H^{k+1}(\cA^{\ast}_0(\cI))$ of the above subcomplex are related
by the map
$$
d: \frac{ \{\omega \in
\Lambda^k : d\omega \in \cA^{k+1}_0(\cI)\}}
{d\Lambda^{k-1}+\cA^{k}_0(\cI)}\longrightarrow \frac{\{\gamma\in
\cA^{k+1}_0(\cI):d\gamma=0\}}{d\cA^{k}_0(\cI)},
$$
which is an isomorphism by the exactness of the de Rham complex of
germs of differential forms on $\mathbb C^n$.

We are interested in $H^n(\cA^{\ast}_0(\cI))$. First notice
the following fact.

\begin{prop}\label{Poincare}
If an ideal $\cI$ in $\cO _n$ has generators
$g_1,\ldots,g_s$, where each $g_i$ is $\cK$-equivalent
to a $\cVnK$-w.q.h. function-germ,
then $H^n(\cA^{\ast}_0(\cI))=0$.
\end{prop}

\begin{rem}
The hypothesis that each $g_i$ is $\cK$-equivalent to
some function-germ that is q.h. for non-negative weights
and total positive weight (and hence $\cVnK$-w.q.h.)
does not require that the map $g=(g_1,\ldots ,g_s)$ is
$\cVnK$-w.q.h. (the source diffeomorphisms in the $\cK$-equivalences
can be different for each $g_i$).
\end{rem}

\begin{proof} It is enough to show that any $n$-form in
$\cI\Lambda^n$  is the differential of a $(n-1)$-form in
$\cI\Lambda^{n-1}$. Let $\omega=\sum_{i=1}^s g_i \omega_i$, where
the $\omega_i$ are $n$-forms. Any $n$-form on $\mathbb C^n$ is
closed and each $g_i=k_i\Phi^{\ast}h_i$, where $k_i$ is a
non-vanishing function-germ, $\Phi_i$ is a diffeomorphism-germ and
$h_i$ is w.q.h. with non-negative weights, at least one of which
is positive. We then apply the following lemma to each
$h_i(\Phi_i^{-1})^{\ast}(k_i\omega_i)$ separately.

\begin{lem}\label{rP}
If $h$ is w.q.h. then for any $n$-form $\omega$ there exists an
$(n-1)$-form $\beta$ such that $h \omega=d(h \beta)$.
\end{lem}

\begin{proof}[Proof of Lemma \ref{rP}]
If $h$ generates the vanishing ideal of $\{h=0\}$ then this is a
corollary of the relative Poincare lemma for varieties that are
quasi-homogeneous with respect to a smooth submanifold
\cite{DJZ2}. More generally (for $\langle h\rangle$ not necessarily
radical) we use the same method as
in the proof of Proposition \ref{prEuler}.

Let $E_{w}$ be (the germ of) the Euler vector field for $h$ and
let $G_t$ be the flow of $E_w$. Then $G_t^{\ast}h=e^{\delta t} h$,
where $\delta$ is quasi-degree of $h$. By direct calculation we
obtain
\begin{equation}
\omega=\int_{-\infty}^{0}(G_t^{\ast}\omega)'dt=d(h\beta),
\end{equation}
where $\beta=\int_{-\infty}^{0}e^{\delta
t}G_t^{\ast}(E_w\rfloor\omega)dt$ is a smooth $(n-1)$-form.
\end{proof}

To conclude the proof of the proposition, we have
from Lemma \ref{rP}
$$
g_i\omega_i=\Phi_i^{\ast}(h_i(\Phi_i^{-1})^{\ast}(k_i\omega_i))=
\Phi_i^{\ast}(d(h_i\beta_i))=d(g_i\alpha_i),
$$
where $\alpha_i=\frac{1}{k_i}\Phi_i^{\ast}\beta_i$.
Hence $\omega=\sum_{i=1}^s g_i
\omega_i=d(\sum_{i=1}^sg_i\alpha_i)$, as desired.
\end{proof}

We can now relate the dimensions of $n$th cohomology
groups of the two complexes in question.

\begin{thm} \label{basic}
For $g_1,\cdots,g_s\in \cI$ we have
$$\dim H^n(\Lambda^{\ast}(\cI))\le \dim \frac{\cO_n}
{\langle g_1,\cdots,g_s,\nabla  g_1,\cdots,\nabla g_s \rangle}+
\dim H^n(\cA^\ast_0(\langle g_1,\cdots,g_s\rangle )).$$
\end{thm}

\begin{proof} For $\cJ :=\langle g_1,\cdots,g_s\rangle \subset\cI$, clearly
$\cJ\Lambda^{n-1}\subset\Lambda^{n-1}(\cI)$, which
implies that
$$\dim
H^n(\Lambda^{\ast}(\cI))=\dim \Lambda^n/d(\Lambda^{n-1}(\cI)) \le
\dim \Lambda^n/d(\cJ\Lambda^{n-1}),$$
where $\dim \Lambda^n/d(\cJ\Lambda^{n-1})=\dim
\Lambda^n/\cA_0^n(\cJ)+\dim \cA_0^n(\cJ)/d(\cJ\Lambda^{n-1})$.
Furthermore, from
$$\cA_0^n(\cJ)=\{\sum_{i=1}^s
g_i\omega_i+dg_i\wedge\sigma_i:\omega_i\in \Lambda^n, \
\sigma_i\in \Lambda^{n-1}, \ i=1,\cdots, s\}$$
we see that $\Lambda^n/\cA_0^n(\cJ)$ is isomorphic to
${\cO_n}/{\langle g_1,\cdots,g_s,\nabla g_1,\cdots,\nabla g_s\rangle }$.
Finally,
$d(\cJ\Lambda^{n-1})=d(\cA_0^{n-1}(\cJ))$ implies that
$\cA_0^n(\cJ)/d(\cJ\Lambda^{n-1})$ and $H^n(\cA _0^\ast(\cJ))$
are equal.
\end{proof}

Theorem \ref{basic} and Proposition \ref{Poincare} imply the
following corollary

\begin{cor}\label{finite-cor}
If $g_1,\cdots,g_s\in\cI$ satisfy the conditions of Proposition
\ref{Poincare} then
$$\dim H^n(\Lambda^{\ast}(\cI))\le
\dim \frac{\cO_n}{\langle \nabla g_1,\cdots,\nabla g_s\rangle }.$$
\end{cor}

\begin{proof} Proposition \ref{Poincare} implies that
$\dim H^n(\cA^\ast_0(\langle g_1,\cdots,g_s\rangle ))=0$, and
$g_i\in \langle \nabla g_i\rangle $
(because $g_i$ is w.q.h.).
\end{proof}

We can now deduce the following finiteness results.

\begin{thm}\label{finite1}
Let $W$ be a variety-germ with an isolated
singularity at $0$.
If the vanishing ideal of $W$  is contained in $\cI$ then $\dim
H^n(\Lambda^{\ast}(\cI))<\infty$.
\end{thm}

\begin{proof}
Let  $\cI(W)$ be generated by $g_1,\cdots,g_s$. Clearly
$g_1,\cdots,g_s\in \cI$ and from Theorem \ref{basic} we have
$$\dim H^n(\Lambda^{\ast}(\cI))\le \dim \frac{\cO_n}
{\langle g_1,\cdots,g_s,\nabla g_1,\cdots,\nabla g_s\rangle }+
\dim H^n(\cA^\ast_0(\cI(W))).$$

From the hypothesis on $W$ we then obtain the finiteness of the
dimensions on the right: $H^n(\cA^\ast_0(\cI(W)))$ is finite by a
result of Bloom and Herrera \cite{B-H} and the colength of
$\langle g_1,\cdots,g_s,\nabla g_1,\cdots,\nabla g_s\rangle$ in
$\cO _n$ is also finite for such $W$ (see our earlier remark).
\end{proof}

\begin{thm}\label{finite2}
Let $\langle g\rangle$ be the vanishing ideal of
a hypersurface having an isolated singularity at $0$.
If $g$ is contained in $\cI$  then $\dim
H^n(\Lambda^{\ast}(\cI))\leq \mu(g)$, where $\mu(g)$ is the Milnor
number of $g$.
\end{thm}

\begin{proof} For $\langle g\rangle\subset\cI$ we obtain from
Theorem \ref{basic}
$$\dim H^n(\Lambda^{\ast}(\cI))\le \dim
\frac{\cO_n}{\langle g,\nabla g\rangle }+
\dim H^n(\cA^\ast_0(\langle g\rangle )).$$
The desired bound then follows from the following
formula of Brieskorn \cite{B} and Sebastiani \cite{Se}:
$\dim H^n(\cA^\ast_0(\langle g\rangle ))=\mu(g)-\tau(g)$, where
$\tau(g):=\dim {\cO_n}/{\langle g,\nabla g\rangle }$ is
the Tjurina number of $g$.
\end{proof}

\begin{rem}\label{finite3}
Theorem \ref{finite1} implies that for a finitely generated
ideal $\cI =\langle g_1,\ldots ,g_p\rangle$ corresponding to
a $\cK$-finite map $f=(g_1,\ldots ,g_p)$ the dimension of
$H^n(\Lambda^{\ast}(\cI))$ is finite dimensional.
For the ideal of an ICIS we have a more precise
bound. For a $\bC$-linear combination
$h=\sum _{i=1}^p a_ig_i$ we have
$\langle h\rangle\subset\cI$, hence
$\dim H^n(\Lambda^{\ast}(\cI))\leq \mu (h)$
(for $\mu (h)<\infty$ we apply Theorem \ref{finite2}, and
otherwise the upper bound is trivial). Furthermore,
for a generic projection $\pi :\bC ^p\to \bC$,
$(y_1,\ldots ,y_p)\mapsto \sum _{i=1}^p a_iy_i$ the
Milnor number of $h=\pi\circ g$, where $g=(g_1,\ldots ,g_p)$,
is finite (recall the usual method for calculating the
Milnor number of an ICIS).
\end{rem}

The above finiteness results can be generalized to the case of
subgroups $H$ of the group of germs of $\mathbb C$-analytic
diffeomorphisms of $\mathbb C^q$. Using the isomorphism
$\theta_q\ni X\mapsto X\rfloor \Omega \in \Lambda^q$ we can prove
in the same way the following result.

\begin{thm} Let $\cJ$ be an ideal in $\cO_q$ generated by
$g_1,\cdots,g_s$. If $\cJ \theta_q$ is contained in $LH$ then
$$\dim \frac{\cO_q}{\Div(LH)}\le \dim
\frac{\cO_q}{\langle g_1,\cdots,g_s,\nabla g_1,\cdots,\nabla
g_s\rangle } +\dim H^q(\cA^\ast_0(\langle g_1,\cdots,g_s \rangle
)).$$
\end{thm}

In particular we obtain the following

\begin{cor}\label{fin-A}
Consider the image $\im f$ of a complex-analytic map-germ
$f:(\mathbb C^n,0)\rightarrow (\mathbb C^p,0)$, and recall
that $L\cA _f^p=Lift (f)$.

(a) If $\im f\subset W$, for some variety-germ $W$ with an
isolated singularity at $0$, then
$\dim \cO_p/\Div(L\cA _f^p)$ is finite.

(b) If $\im f\subset g^{-1}(0)$, for some hypersurface
germ $g^{-1}(0)$ with an
isolated singularity at $0$, then
$\dim \cO_p/\Div(L\cA _f^p)\leq \mu (g)$.
\end{cor}

\begin{rem}\label{rem-fin-A}
Suppose that $f:(\mathbb C^n,0)\rightarrow (\mathbb C^p,0)$
is an $\cA$-finite map-germ with target dimension $p\ge 2n$.
Then $\im f$ is a variety-germ with an isolated singularity at $0$,
hence $\cA$-finiteness implies $\cVpA$-finiteness (in the
sense that the moduli-space $\cM (\cVpA ,f)$ is finite
dimensional, by the above result). This generalizes the
corresponding result in \cite{IJ} for plane curves.

Also, for map-germs
$f:(\mathbb C^n,0)\rightarrow (\mathbb C^2,0)$, $n\ge 2$,
for which $L\cA _f^p=Lift (f)$ is equal to Derlog of the
discriminant we have that the
$\cA$-finiteness of $f$ implies the $\cVpA$-finiteness
(notice, the discriminant is a curve with isolated singularities).

For $p<2n$ the image (for $n<p$) or the discriminant
(for $n\ge p\ge 3$) of an $\cA$-finite singular map-germ $f$
in general has non-isolated singularities
(except perhaps for a generalized fold map $f$).
Hence the above finiteness
result cannot be applied.
\end{rem}

\section{The foliation of $\cA$-orbits by $\cVpA$-orbits}\label{fol-A}

In this section we study the foliation of $\cA$-orbits
of map germs $f:(\bK ^n,0)\to (\bK ^p,\Omega _p,0)$
by $\cVpA$-orbits.
Our main objective here is the classification of
$\cVpA$-simple orbits inside the $\cA$-simple orbits, and
in dimensions $(n,2)$ and $(n,2n)$, $n\ge 2$, we give
explicit lists (see \S\ref{simple(n,2)} and \S\ref{simple(n,2n)}).
We also consider $\cVpA$-orbits of positive modality that
are s.q.h. but not w.q.h (see \S\ref{sqh}) and w.q.h. multigerms
(see \S\ref{multigerms}).

For the pairs $(n,p)$ for which the $\cA$-simple orbits are
known -- i.e., for $n\ge p$, $(1,p)$ any $p$, $p=2n$, $(2,3)$
(any corank) and for $(3,4)$ (of corank 1), see the references
below -- we find that:

\begin{enumerate}
\item an $\cA$-simple germ is
$\cVpA$-simple if and only if it does not lie in the closure of the
orbit of any non-weakly quasihomogeneous germ,
\item for $n<2p$ and for $p=2n$, with $n\le 3$,
an $\cA$-simple germ is
$\cVpA$-simple if and only if it does not lie in the closure of the
orbit of any non-quasihomogeneous germ.
\end{enumerate}

(The classifications of $\cA$-simple orbits can be found in
the following papers: $(n,p)=(1,2)$ \cite{B-G}, $(1,3)$ \cite{G-H},
$(1,p)$ ($p\ge 3$) \cite{ACurves}, $(n,2n)$ ($n\ge 2$) \cite{KPR},
$(2,3)$ \cite{Mond}, $(3,4)$ \cite{H-K}, $(n,2)$ ($n\ge 2$)
\cite{Ri,RiRu} and $(3,3)$ \cite{MT}. The survey in \cite{Gor}
describes the simple singularities of projections of complete
intersections, this {\em a priori} finer classification corresponds
to the $\cA$-classification for $n\ge p$.)

After explaining the techniques for verifying the above claim,
we will describe two particular cases in detail. First,
the classification of $\cVpA$-simple orbits in dimensions
$(n,2)$, $n>1$, because for $p=2$ the volume preserving
and the symplectic classifications agree. Combining this
classification with the one by Ishikawa and Janeczko \cite{IJ}
for curves (i.e., for $(1,2)$) yields all simple map-germs
into the symplectic plane. And second, the
classification of $\cVpA$-simple orbits in dimensions
$(n,2n)$, where ``non-trivial'' weakly quasihomogeneous
germs (that are
not quasihomogeneous nor ``trivially w.q.h.'') start
appearing.

Notice that the condition w.q.h. (for $\cVpA$)
in Proposition \ref{inf-wqh} and Theorem \ref{orb-wqh}
is a sufficient condition for the absence of $\cVpA$-moduli,
we do not know whether it is necessary.
However, for all $\cA$-simple germs in the dimension ranges
$(n,p)$ in which the $\cA$-simple classification is known
(see above) the condition w.q.h. is necessary and sufficient
for the absence of $\cVpA$-moduli. This obviously implies
the criterion above: an $\cA$-simple germ $f$ is $\cVpA$-simple
if and only if $f$ is only adjacent to w.q.h. germs.
All known examples of $\cA$-simple map-germs $f$ that fail
to be w.q.h. are of the form $f=f_0+h$, where
$f_0$ is quasihomogeneous, $h$ is a monomial vector of positive
filtration (weighted degree) and the restriction of
$\gamma _{f_0}:L\cA\to L\cA\cdot f$ to the filtration-0
parts (of the filtered modules in source and target)
has 1-dimensional kernel. In this situation the
coefficient of $h$ is a modulus for $\cVpA$
(see Lemma \ref{ker-gamma} below).

Consider $L\cVpA\cdot f\subset L\cA\cdot f= tf(\cM _n\cdot\theta
_n)+wf(\cM _p\cdot\theta _p)$. For the subgroup $\cVpA
=\cR\times\cVpL$ of $\cA$ we have to restrict the homomorphism
$wf:\theta _p\to\theta _f$, $wf(b)=b\circ f$ to divergence free
vector fields $b$, hence $L\cVpL \cdot f$ is no longer a
$C_p$-module. Let $\Lambda _d$ denote the $\bK$-vector space of
homogeneous divergence free vector fields in $\bK ^p$ of degree
$d$. Notice that $\Lambda _d$ is the kernel of the epimorphism
\[
\Div : (\theta _p)_{(d)}:=
\frac{\cM _p^d\cdot\theta _p}{\cM _p^{d+1}\cdot\theta _p}
\to H_{(d-1)}:= \frac{\cM _p^{d-1}}{\cM _p^d},
\]
which maps a vector field on $\bK ^p$ of degree $d$ to its
divergence. Hence
\[
\dim\Lambda _d = \dim (\theta _p)_{(d)}-\dim H_{(d-1)}=
(p-1)\left(^{p+d-1}_{\ \ \ d}\right) + \left(^{p+d-2}_{\ \ \ d
}\right).
\]
The $\dim\Lambda _d$ vector fields
\[
\prod _{l\ne i}y_l^{\al _l}\partial /\partial y_i,
~~\sum _l\al _l=d,~~i=1,\ldots ,p
\]
and (setting $h_{y_i}:=\partial h/\partial y_i$)
\[
-h_{y_j}\partial /\partial y_1+h_{y_1}\partial /\partial y_j,
~~h=\prod _l y_l^{\al _l},~~\al _1,\al
_j\ge 1,\sum _l\al _l=d+1,~~j=2,\ldots ,p
\]
are clearly linearly independent and hence form a basis for
$\Lambda _d$. The tangent space to the $\cVpL$-orbit at $f$ is then
given by $L\cVpL\cdot f=f^{\ast}\oplus _{d\ge 1}\Lambda _d$.

The criterion in the next easy lemma is sufficient for detecting
in the existing classifications of $\cA$-simple orbits those which
are foliated by an $r$-parameter family, $r\ge 1$, of
$\cVpA$-orbits.

\begin{lem}\label{ker-gamma}
Consider a map-germ $f_u:(\bK ^n,0)\to (\bK ^p,0)$ of the form
$f_u=f+u\cdot M$, where $f$ is a quasi-homogeneous germ, $u\in\bK$
and $M=X^{\al}\cdot\partial /\partial y_j\notin L\cA\cdot
f=L\cVpA\cdot f$ is a monomial vector of positive weighted degree
(with respect to the weights of $f$). Then we have the following:

(i) The coefficient $u$ is not a modulus for $\cA$-equivalence.

(ii) For a set of weights for which $f$ is weighted homogeneous,
let $(\theta _n)_0$, $(\theta _p)_0$ and $(\theta _f)_0$ denote
the filtration-0 parts of the modules of source-, target-vector
fields and vector fields along $f$, respectively. If the kernel of
the linear map
\[
\g _f:  (\theta _n)_0 \oplus (\theta _p)_0 \to (\theta _f)_0, ~~~
(a,b)\mapsto tf(a)-wf(b),
\]
of $\bK$-vector spaces is 1-dimensional then $u$ is an
$\cVpA$-modulus of $f_u$.
\end{lem}

\begin{proof}
Let $f$ be weighted-homogeneous for the weights $w_1,\ldots ,w_n$,
and associate to the target variables the weights
$\de _1,\ldots ,\de _p$.  Then the weighted degree of
$\partial /\partial y_i$ is $-\de _i$ so that $f$ has
filtration 0 and $M$ has filtration $r>0$.

For $\cA$-equivalence we consider the following element of
$L\cA\cdot f_u$:
$$
tf_u(\sum _{i=1}^n w_ix_i\cdot \partial /\partial x_i)-
wf_u(\sum _{j=1}^p \de _j y_j\cdot \partial /\partial y_j)=r
uM.
$$

From Mather's lemma (Lemma 3.1 in \cite{MaIV}) we conclude that
the connected components of $\bK\setminus\{ 0\}$ of the parameter
axis lie in a single $\cA$-orbit, hence $u$ is not a modulus for
$\cA$.

For the second statement we observe that $\dim\ker\g _f=1$
implies that this kernel is spanned by the pair of Euler
vector fields $(E_w,E_{\de})$ in source and target
(which is unique up to a multiplication by an
element of $\bK ^{\ast}$). And $M\notin L\cA\cdot f$ implies that
the only generator of $M$ in $L\cA\cdot f_u$ must be of the form
$tf_u(a)-wf_u(b)$ with $(a,b)$ a non-zero multiple of
$(E_w,E_{\de})$. But $E_{\de}$ has non-zero divergence, hence this
generator does not belong to $L\cVpA\cdot f_u$. Now Mather's lemma
implies that $u$ is a modulus for $\cVpA$.
\end{proof}

\subsection{$\cVpA$-simple, hence symplectically simple, maps from
$n$-space to the plane}\label{simple(n,2)}

The following classification, in combination with
Ishikawa and Janeczko's classification of plane curves \cite{IJ},
provides a complete list of simple map-germs into the plane $\bC ^2$,
up to source diffeomorphisms and target symplectomorphisms
(volume preserving diffeomorphisms of $\bC ^2$ are
symplectomorphisms).

\begin{prop}\label{class(n,2)}
Any $\cVpA$-simple map-germ $f:(\bC ^n,0)\to (\bC ^2,0)$,
$n\ge 2$, is
equivalent to one of the following normal forms
(here $Q=\sum _{i=1}^{n-2}z_i^2$ for $n>2$ and $Q=0$ for $n=2$):
$(x,y)$; $(x,y^2+Q)$; $(x,xy+y^3+Q)$;
$(x,y^3+x^ky+Q)$, $k>1$; $(x,xy+y^4+Q)$.
\end{prop}

\begin{proof}
Any $\cA$-simple germ in dimensions $(n,2)$, $n\ge 2$,
which does not appear in the above list, is adjacent
to one of the following germs (for $n>2$, up to a suspension
by $Q$ defined above): $(x,xy+y^5+y^7)$,
$(x,xy^2+y^4+y^5)$ or $(x^2+y^3,y^2+x^3)$ (see the adjacency
diagrams in \cite{Ri} and \cite{RiRu}). These three germs
fail to be weakly quasihomogeneous and they satisfy
the hypotheses of Lemma \ref{ker-gamma}, hence they
have at least one modulus for $\cVpA$. In fact, the
parameter $a$ in
$(x,xy+y^5+ay^7)$, $(x,xy^2+y^4+ay^5+\ldots )$ and
$(x^2+ay^3,y^2+x^3)$ is a modulus for $\cVpA$.
\end{proof}

\subsection{$\cVpA$-simple maps from $n$-space to $2n$-space}
\label{simple(n,2n)}

In the same way we obtain the $\cVpA$-simple germs
in dimensions $(n,2n)$, $n\ge 2$ (notice that $n=1$ again
corresponds to the classification in \cite{IJ}). Except
for the appearance of a series of w.q.h. germs (see the last two
normal forms in Proposition \ref{class(n,2n)} below,
corresponding to type $22_k$ and $23$ in \cite{KPR}),
which are not q.h. nor trivially w.q.h.,
this classification follows from the classification
of $\cA$-orbits (and some information about adjacencies
between these orbits) in \cite{KPR}, using the same arguments
as in dimensions $(n,2)$. The classifications in dimensions
$(2,4)$ and $(n,2n)$, $n\ge 3$, are as follows.

\begin{prop}\label{class(2,4)}
Any $\cVpA$-simple map-germ $f:(\bC ^2,0)\to (\bC ^4,0)$ is
equivalent to one of the following normal forms:
$(x,y,0,0)$;
$(x,xy,y^2,y^{2k+1})$, $k\ge 1$;
$(x,y^2,y^3,x^ky)$, $k\ge 2$;
$(x,y^2,y^3+x^ky,x^ly)$, $l>k\ge 2$;
$(x,y^2,x^2y+ y^{2k+1},xy^3)$, $k\ge 2$;
$(x,y^2,x^2y,y^5)$;
$(x,y^2,x^3y+y^5,xy^3)$;
$(x,xy,xy^2+y^{3k+1},y^3)$, $k\ge 1$;
$(x,xy,xy^2+y^{3k+2},y^3)$, $k\ge 1$;
$(x,xy+y^{3k+2},xy^2,y^3)$, $k\ge 1$;
$(x,xy,y^3,y^4)$;
$(x,xy,y^3,y^5)$.
\end{prop}

\begin{prop}\label{class(n,2n)}
Any $\cVpA$-simple map-germ $f:(\bC ^n,0)\to (\bC ^{2n},0)$,
$n\ge 3$, is
equivalent to one of the following normal forms (here
$\bx$ denotes $x_1,\ldots ,x_{n-1}$, and notice that the last
two normal forms are only $\cVpA$-simple for $n\ge 4$):\newline
$(\bx,y,0,\ldots ,0)$\newline
$(\bx ,x_1y,\ldots ,x_{n-1}y,y^2,y^{2k+1})$, $k\ge 1$\newline
$(\bx ,x_2y,\ldots ,x_{n-1}y, y^2,y^3,x_1^ky)$, $k\ge 2$\newline
$(\bx ,x_2y,\ldots ,x_{n-1}y,y^2,y^3+x_1^ky,x_1^ly)$, $l>k\ge 2$
\newline
$(\bx ,x_2y,\ldots ,x_{n-1}y,y^2,x_1^2y +y^{2k+1},x_1y^3)$, $k\ge 2$
\newline
$(\bx ,x_2y,\ldots ,x_{n-1}y,y^2,x_1^2y,y^5)$\newline
$(\bx ,x_2y,\ldots ,x_{n-1}y,y^2,x_1^3y+y^5,x_1y^3)$\newline
$(\bx ,x_3y,\ldots ,x_{n-1}y,y^2,x_1^2y,x_2^2y,y^3+x_1x_2y)$\newline
$(\bx ,x_3y,\ldots ,x_{n-1}y,y^2,x_1^2y,x_2^2y,y^3)$\newline
$(\bx ,x_3y,\ldots ,x_{n-1}y,y^2,x_1x_2y,(x_1^2+x_2^3)y,y^3+ x_2^2y)$
\newline
$(\bx ,x_3y,\ldots ,x_{n-1}y,y^2,x_1x_2y,(x_1^2+x_2^3)y,y^3+x_2^3y)$\newline
$(\bx ,x_3y,\ldots ,x_{n-1}y,y^2,x_1x_2y,(x_1^2+x_2^3)y,y^3)$ \newline
$(\bx ,\bx y,x_1y^2+y^{3k+1},y^3), k\ge 1$\newline
$(\bx ,\bx y,x_1y^2+y^{3k+2},y^3), k\ge 1$\newline
$(\bx ,x_1y+y^{3k+2},x_2y,\ldots ,x_{n-1}y,x_1y^2,y^3), k\ge 1$\newline
$(\bx ,x_1y,x_2y+y^{3k+2},x_3y,\ldots ,x_{n-1}y,x_1y^2+y^{3l+1},y^3),
l>k\ge 1$\newline
$(\bx ,x_1y,x_2y+y^{3k+2},x_3y,\ldots ,x_{n-1}y,x_1y^2+y^{3l+2},y^3),
l>k\ge 1$\newline
$(\bx ,x_1y+y^{3l+2},x_2y+y^{3k+2},x_3y,\ldots ,x_{n-1}y,x_1y^2,y^3),
l>k\ge 1$\newline
$(\bx ,\bx y,y^3,y^4)$\newline
$(\bx ,\bx y,y^3,y^5)$\newline
$(\bx ,x_1y+y^3,x_2y,\ldots ,x_{n-1}y,x_1y^2+y^{2k+1},x_2y^2+y^4)$,
for $k=2$ and $n\ge 4$\newline
$(\bx ,x_1y+y^3,x_2y,\ldots ,x_{n-1}y,x_1y^2+y^5,y^4)$,
for $n\ge 4$.
\end{prop}

\begin{proof}
Except for the germs of type $22_k$ and 23 in dimensions $(n,2n)$, $n\ge 4$
(these are the last two germs in the second list above),
all $\cA$-simple germs in \cite{KPR} are either quasihomogeneous
or they satisfy the hypotheses of Lemma \ref{ker-gamma} and
hence have at least one $\cVpA$-modulus.

Consider, then, the series $22_k$ of map germs
$(\bC ^n,0)\to (\bC ^{2n},0)$, $n\ge 3$ given by:
\[
g_k=(x_1,\ldots ,x_{n-1},x_1y+y^3,x_2y,\ldots
,x_{n-1}y,x_1y^2+y^{2k+1}, x_2y^2+y^4), k\ge 2.
\]
The germs $22_k$ are not semi-quasihomogeneous: if we write
$g_k=f+y^{2k+1}\cdot e_{2n-1}$ then the weighted homogeneous
initial part $f$ is not $\cA$-finite. For $n=3$ all the germs
$22_k$ are $\cA$-simple, for $n\ge 4$ only $22_2$ is $\cA$-simple
(the germs $22_{\ge 3}$ do  not have an $\cA$-modulus, but they
lie in the closure of non-simple $\cA$-orbits), see \cite{KPR}.

Now consider $\cVpA$-equivalence. Writing $f_u=f+u\cdot
y^{2k+1}\cdot e_{2n-1}$ we see that $\dim\ker\gamma _f=n-2$. For
$n=3$ part (ii) of Lemma \ref{ker-gamma} therefore implies that the
coefficient $u$ is an $\cVpA$-modulus. For $n\ge 4$ the germs $f_u$
are weakly quasi-homogeneous (take weights $w(x_1)=w(x_2)=w(y)=0$ and
$w(x_i)=1$, $i\ge 3$) and $\cVpA$-equivalent to $g_k$
(for $u\ne 0$).

For the germ of type $23$ the argument is the same.
\end{proof}

\subsection{Semi-quasihomogeneous, but not weakly
quasihomogeneous, singularities}\label{sqh}

Non-w.q.h. maps
have a decomposition $f=f_0+h$ with $f_0$ q.h.
and $h$ of positive degree (relative to the weights of $f_0$).
The normal space $N\cA\cdot f_0:=\cM _n\cdot\theta _{f_0}/L\cA\cdot f_0$
decomposes into a part of non-positive filtration and a part of
positive filtration, denoted by $(N\cA\cdot f_0)_+$.
Using the fact that $L\cVpA\cdot f_0=L\cA\cdot f_0$ and Mather's
lemma we obtain the following formal pre-normal form
for an element of an $\cVpA$-orbit inside $\cA\cdot f$:
\[
f'=f_0+\sum _{h_i\in B(f_0)_+} a_i h_i,
\]
where $B(f_0)_+$ denotes a base for $(N\cA\cdot f_0)_+$ as
a $\bK$-vector space.
Notice that for semi-quasihomogeneous maps $f$ the above
sum is finite (because $f_0$ is $\cA$-finite), otherwise it
is infinite.

Preliminary empirical examples indicate that in the s.q.h.
case (where $f_0$ is $\cA$-finite) the above pre-normal for
$f'$ is in fact a (formal) normal form for $\cVpA$. In this case
the coefficients $a_i$ are independent moduli for $\cVpA$
(some $a_i$ might also be moduli for $\cA$). If this observation
holds in general for s.q.h. maps in dimensions $(n,p)$, $n\ge p-1$,
(and Conjecture I in \cite{CMWA} is true) then
such maps $f$ satisfy the formula
\[
\cod (\cA _{\Omega _p,e},f)=\mu _{\Delta}(f)
\]
as pointed out in the introduction (here $\mu _\Delta$ denotes
the discriminant Milnor number (for $n\ge p$) or the
image Milnor number (for $p=n+1$)). Also notice that
for $n\ge p$ we have $\cod (\cA _{\Omega _p,e},f)\leq \mu _{\Delta}(f)$,
independent of the correctness of the above conjectures.

Let us consider some examples in dimensions $(n,2)$, $n\ge 2$.

\begin{ex} The $\cA$-simple non-w.q.h. germs in dimensions
$(n,2)$ have the following formal normal forms for $\cVpA$
(the normal forms $(*)$ are not s.q.h. and $Q$ denotes a sum
of squares in additional variables):
$(x,xy+y^5+ay^7+Q)$;
$(x,xy^2+y^5+ay^6+by^9+Q)$;
$(x,x^2y+y^4+ay^5+Q)$;
$(*)$ $(x,xy^2+y^4+\sum _{k\ge 2}a_ky^{2k+1}+Q)$;
$(*)$ $(x^2+ay^{2l+1},y^2+x^{2m+1})$, $l\ge m\ge 1$.

The first three normal forms $f$ are s.q.h. and their
$\cA _{\Omega _p,e}$-codimensions are equal to the
$\cA _e$-codimensions of their initial parts $f_0$,
and these are given by 3, 5 and 4, respectively.
And from \cite{DM} we have the formula
$\mu _\Delta (f)=\mu (\Sigma _f)+d(f)$
(relating the discriminant Milnor to the Milnor number
of the critical set and the double-fold number), which
gives for the three normal forms $3=0+3$, $5=1+4$ and
$4=2+2$, respectively.

The two series of non-s.q.h. maps $f$ (marked by $(*)$)
are GTQ in the sense of
\cite{BdPW} and the Milnor numbers of their discriminant
curves $\Delta _f$ (not to be confused with the discriminant
Milnor numbers of $f$) are $2k+7$ and $2(l+m)+3$, respectively.
These Milnor numbers are upper bounds for the
$\cVpA$-moduli space of $f$ (by Remark \ref{rem-fin-A}).
Formal calculations (at the infinitesimal level using
Mather's lemma) actually show that
$\dim \cM(\cVpA ,(x^2+y^{2l+1},y^2+x^{2m+1}))=1$,
modulo $\cM _n^{\infty}\theta _f$,
and that we can take the above (formal) normal form
for $\cVpA$-equivalence with the parameter $a$ as the
modulus. For the other
non-s.q.h. map we only know that $a_2$ is a modulus and
that we can take $a_3=a_4=0$ (provided $a_2\neq 0$),
for $a_k$, $k>4$, the corresponding calculations of
$L\cVpA\cdot f+\cM _n^{k+1}\theta _f$
seem very tedious.

Finally, a brief remark on our computation of $\mu (\Delta _f)$
for the above two series. We use the formulas
$2\delta =\mu +r-1$ (relating the $\delta$-invariant, the
number of branches $r$  and $\mu$ of a planar curve-germ)
and $\delta (\Delta _f)=c(f)+d(f)+\delta (\Sigma _f)$
(where $c(f)$ and $d(f)$ are the numbers of cusps
and double folds, respectively, in a stable perturbation
$f_t$ of $f$, hence $\delta (\Delta _{f_t})=c(f)+d(f)$).
For $f=(x,xy^2+y^4+y^{2k+1})$ we obtain
$\delta (\Delta _f)=3+k+1$ (see Table 1 in \cite{Ri}),
hence $\mu (\Delta _f)=2k+7$
(notice that the discriminant has $r=2$ branches).
This contradicts the claim in part (c) of example 1
in \cite{BdPW} that $\Delta _f$ has an $E_{6k+1}$ singularity.
\end{ex}

\begin{ex} The $\cA$-unimodal germs in dimensions
$(n,2)$ lie in the closure of the orbits of
one of the following $\cA$-unimodal
s.q.h. germs (see \cite{Ri2}, and $Q$ is again a sum of
squares in additional variables):
\[
(x,y^4+x^3y+ax^2y^2+x^3y^2+Q),~~a\neq -3/2
\]
\[
(x,xy+y^6+y^8+ay^9+Q)
\]
\[
(x,xy+y^3+ay^2z+z^3+z^5+Q).
\]
For $\cVpA$-equivalence the corresponding normal forms $f$ are:
\[
(x,y^4+x^3y+ax^2y^2+bx^3y^2+Q)
\]
\[
(x,xy+y^6+ay^8+by^9+cy^{14}+Q)
\]
\[
(x,xy+y^3+ay^2z+z^3+bz^5+Q).
\]
All $\cA$-unimodal germs therefore have $\cVpA$-modality at least
two. Also, for the above $f=f_0+h$ we again have
$\cod (\cA _{\Omega _p,e},f)=\cod (\cA _e,f_0)=\mu _{\Delta}(f)$.
\end{ex}

\subsection{Weakly quasihomogeneous multigerms}
\label{multigerms}

Before leaving the subject of $\cVpA$-classification we
make a final remark.
All the results on $\cVpA$-equivalence can be easily extended
to multigerms $f=(f^1,\ldots ,f^s):(\bK ^n,S)\to (\bK ^p,\Omega _p,0)$
at an $s$-tuple $S=\{ q^1,\ldots ,q^s\}\subset\bK ^n$ of points
in the source. Such an $f$ is $\cVpA$-w.q.h. if each
component $f^i=(f_1^i,\ldots ,f_p^i)$ is $\cVpA$-w.q.h. as a monogerm
for (possibly different) sets of weights $\{ w_1^i,\ldots ,w_n^i\}$
but of the same weighted degrees $\deg f_j^1=\ldots =\deg f_j^s=\delta _j$,
$j=1,\ldots ,p$. Also, if the above weights $w_j^i$ are positive
integers then we say that $f$ is q.h. as a multigerm.

Using Mather's \cite{MaV} characterization of $\cA$-stability
of multigerms in terms of multitransversality to $\cK$-orbits
of multigerms, it is not hard to see that all $\cA$-stable
multigerms are q.h. and hence $\cVpA$-w.q.h., which implies
that the classifications of $\cA$-stable and $\cVpA$-stable
orbits (over $\bC$) also agree for multigerms.

\section{The foliation of $\cK$-orbits by $\cVnK$- and
$\cVpK$-orbits}\label{fol-K}

In this section we consider the volume-preserving versions
of the classification of ICIS or,  in other words, of
$\cK$-finite maps $f:\bC ^n\to\bC ^p$, $n\ge p$. Recall
that all $\cK$-simple $f$ and all $f$ whose differential
has non-zero rank are w.q.h. for both $\cVnK$ and
$\cVpK$. Hence we will consider $\cK$-unimodal germs $f$
of rank 0 and concentrate on the more interesting
group $\cVnK$ (the condition w.q.h. for $\cVpK$ is weaker
than that for $\cVnK$, hence
$\dim \cM ( \cVnK ,f)=0$ implies $\dim \cM ( \cVpK ,f)=0$).
The relevant $\cK$-classifications are therefore those
in dimensions $(n,p)=(3,2)$ and $(4,2)$ (see \cite{Wa2})
and $(2,2)$ (see \cite{DG1}) and $(3,3)$ (see \cite{DG2}).
Recall that the $\cVnK$-classification of hypersurfaces
$f^{-1}(0)$ has been settled by the result of Varchenko \cite{Va1},
which gives $\dim \cM ( \cVnK ,f)=\mu (f)-\tau (f)$.

Looking at the lists in \cite{Wa2,DG1,DG2} we see
(using our results) that
a $\cK$-unimodal map-germ $f$ is w.q.h. for $\cVnK$
if and only if it is quasihomogeneous. We can therefore
state:
\begin{enumerate}
\item[(1)] A $\cK$-unimodal map-germ $f$ of rank 0
is $\cVnK$-unimodal if and only if it is q.h. and
does not lie in the closure of a non-q.h. $\cK$-orbit.
\item[(2)] For a $\cK$-unimodal map-germ $f$ of rank 0
such that $f^{-1}(0)$ defines a ICIS of positive dimension
and of codimension greater than one we have the following:
$(i)$ $f$ is q.h. if and only if $\cM (\cVnK ,f)=0$, and
$(ii)$ for $f$ non-w.q.h. the dimension of $\cM (\cVnK ,f)$ is
one or two.
\item[(3)] For map-germs $f$ of positive rank we recall
that the $\cK$- and $\cVnK$-classifications agree.
\end{enumerate}

We will now apply our finiteness results for
$\cM ( \cVnK ,f)\cong H^n(\Lambda ^*(f^*\cM _p))$
to some examples of non-q.h. (and non-w.q.h.) map-germs
$f$ from the classifications in \cite{Wa2,DG1,DG2}.
These results give upper bounds for
$\dim\cM ( \cVnK ,f)$, and for certain $f$ some of these upper
bound will coincide with the following lower bound
(which is analogous to Lemma \ref{ker-gamma} in the $\cVpA$ case).

\begin{lem}\label{ker-gamma-K}
Consider a map-germ $f_u:(\bC ^n,0)\to (\bC ^p,0)$ of the form
$f_u=f+u\cdot M$, where $f$ is a quasi-homogeneous germ, $u\in\bC$
and $M=X^{\al}\cdot\partial /\partial y_j\notin L\cK\cdot
f=L\cVnK\cdot f$ is a monomial vector of positive weighted degree
(with respect to the weights of $f$).
For a set of weights for which $f$ is weighted homogeneous,
let $(\theta _n)_0$, $(gl_p(\cO _n))_0$ and $(\theta _f)_0$ denote
the filtration-0 parts of the relevant modules.
If the kernel of the linear map
\[
\g _f:  (\theta _n)_0 \oplus (gl_p(\cO _n))_0 \to (\theta _f)_0, ~~~
(a,B)\mapsto tf(a)-B\cdot f,
\]
of $\bK$-vector spaces is 1-dimensional then $u$ is an
$\cVnK$-modulus of $f_u$. Hence the dimension of
$\cM ( \cVnK ,f_u)$ is positive.
\end{lem}

In our first example we consider positive dimensional
complete intersections, defined by $\cK$-finite maps $f$,
that are not hypersurfaces. In all our examples we have
(for positive dimensional) ICIS that
$\dim\cM (\cVnK ,f)\leq \mu (f) -\tau (f)$, and for
a s.q.h. germ $f=f_0+h$ this inequality holds in general.
(For such $f=f_0+h$ we have
$\dim\cM (\cVnK ,f)=\dim\cM (\cK _{\Omega _{n},e},f)
\leq \tau (f_0)-\tau (f)$ and
$\tau (f_0)=\mu (f_0)=\mu (f)$.)
The germs $f$ in the example (which are not s.q.h.) show
that this inequality can be strict.

\begin{ex}\label{FW}
Consider the $\cK$-unimodal space-curves $FW_{1,i}$
from \cite{Wa2}, given by
\[
f=(g_1,g_2)=(xy+z^3,xz+y^2z^2+y^{5+i}), i>0.
\]
Writing $f=f_0+(0,y^{5+i})$, where $f_0$ is q.h. for
$w=(7,2,3)$, $\delta =(9,10)$ and where $(0,y^{5+i})$
has filtration $2i>0$, and applying Lemma \ref{ker-gamma-K}
we have $\dim\cM ( \cVnK ,f)\ge 1$.
The component functions of $f$ are both q.h. (for weights
$w_1=(1,2,1)$ and $w_2=(7+i,2,3+i)$, respectively)
and $\cO _3/\langle \nabla g_1,\nabla g_2\rangle\cong\bC$,
hence $\dim\cM ( \cVnK ,f)\le 1$ by Corollary \ref{finite-cor}.
Therefore $\dim\cM ( \cVnK ,f)= 1$ and the family
$f_a=(xy+z^3,xz+y^2z^2+ay^{5+i})$ parameterizes
the $\cVnK$-orbits inside $\cK\cdot f$.

A weaker upper bound for $\dim\cM ( \cVnK ,f)$ follows
from Remark \ref{finite3} (which does not require that
the component functions are w.q.h.):
take a projection $\pi$ onto the first target coordinate,
then $\pi\circ f=g_1$ and $\dim\cM ( \cVnK ,f)\leq\mu (g_1)=2$.

Finally, notice that $\dim\cM ( \cVnK ,f)=1$ is
smaller than the difference of $\mu (f)=16+i$ and $\tau (f)=14+i$,
where $\tau (f)$ denotes the dimension of $T^1_f=N\cK _e\cdot f$.
Recall that for hypersurfaces $h^{-1}(0)$ we
have $\dim\cM ( \cVnK ,h)=\mu (h)-\tau (h)$ (by \cite{Va1}),
in all our examples of higher codimensional ICIS $g^{-1}(0)$
we have $\dim\cM ( \cVnK ,g)\leq \mu (g)-\tau (g)$. Notice
that for a ``suspension'' $G=(z,g)$ of $g$ ($z$ an extra variable)
$\mu (G)-\tau (G)=\mu (g)-\tau (g)$, but $dG(0)$ has positive
rank hence $\dim\cM ( \cVnK ,G)=0$, the difference
between both sides of the inequality above can therefore
be arbitrarily large. But the examples $f=FW_{1,i}$ show
that even in the rank 0 case the Varchenko formula does not
hold for ICIS of codimension greater than 1.

Also notice that the series $FW_{1,i}$, $i>0$, lies in
the closure of the $\cK$-orbit of the s.q.h. germ
\[
g_\lambda =(xy+z^3,xz+y^2z^2+\lambda y^{5}+y^6), \lambda\neq 0,-1/4,
\]
where $\mu (g_\lambda )=16$ and $\tau (g_\lambda )=15$.
Omitting the higher filtration $y^6$-term we obtain
type $FW_{1,0}$ in Wall's list \cite{Wa2}, which is q.h.
and $\mu (FW_{1,0})=\tau (FW_{1,0})=16$. Notice
that $FW_{1,1}$ (with $\mu (F_{1,1})=17$ and $\tau (F_{1,1})=15$)
corresponds to the exceptional parameter $\lambda =0$ in
the modular stratum
$\bigcup _{\lambda\in\bC\setminus\{ 0,-1/4\}} \cK\cdot g_\lambda$
(which seems to be missing in Wall's list) and does not
lie in the closure of the orbit of $FW_{1,0}$.
\end{ex}

\begin{ex}\label{ex(3,3)}
Consider the $\cK$-unimodal equidimensional maps of type
$h_{\lambda ,q}$ from \cite{DG2}, given by
\[
f=f^{\lambda}:=(xz+xy^2+y^3,yz,x^2+y^3+\lambda z^q)=
f_0+(0,0,y^3+\lambda z^q), q>2.
\]
The initial part $f_0$ is q.h. of type $w=(1,1,2)$, $\delta =(3,3,2)$
and $fil(0,0,y^3)=1$, $fil(0,0,z^q)=2(q-1)>1$. Applying
Lemma \ref{ker-gamma-K} to $f'=f_0+(0,0,ay^3)$ we see that
$a$ is a $\cVnK$-modulus of $f'$ and hence of $f$,
hence $\dim\cM ( \cVnK ,f)\ge 1$. The component functions
of $f=(g_1,g_2,g_3)$ are q.h. for distinct sets of weights, namely
for $w_1=(1,1,2)$, any $w_2$ and $w_3=(3q,2q,6)$.
Now $\cO _3/\langle \nabla g_1,\nabla g_2,\nabla g_3\rangle\cong\bC$,
so that $\dim\cM ( \cVnK ,f)=1$ (by Corollary \ref{finite-cor}
and the above lower bound -- the upper bound also follows
from $\mu (g_1+g_2+g_3) =1$, by Remark \ref{finite3}).
And for each $f^{\lambda}=(xz+xy^2+y^3,yz,x^2+y^3+\lambda z^q)$,
$\lambda\in\bC$, the family
$f^{\lambda}_a=(xz+xy^2+y^3,yz,x^2+ay^3+\lambda z^q)$ parameterizes
the $\cVnK$-orbits inside $\cK\cdot f^{\lambda}$.
\end{ex}

\begin{ex}\label{ex(2,2)}
Finally, consider the $\cK$-unimodal equidimensional maps of type
$G_{k,l,m}$ from \cite{DG1}, given by
\[
f=(g_1,g_2)=(x^2+y^k,xy^l+y^m)=f_0+(0,y^m),
\]
where $k\neq 2(m-l)$ and either $k\leq l$, $l+1<m<l+k-1$
(case $(a)$) or $l<k<2l-1$, $k<m<2l$ (case $(b)$).
As above we check that the coefficient of $(0,y^m)$
is a $\cVnK$-modulus, hence $\dim\cM ( \cVnK ,f)\ge 1$.
And again the $g_i$ are q.h. for distinct sets of weights,
but now
$\cO _2/\langle \nabla g_1,\nabla g_2\rangle\cong
\bC\{ 1,y,\ldots ,y^r\}$,
where $r=k-1$ in case $(a)$ and $r=l$ in case $(b)$.
Hence $1\leq \dim\cM ( \cVnK ,f)\leq r$.

We can also obtain an upper bound using Remark \ref{finite3}:
take the generic projection $\pi$ onto the first target coordinate,
then $g_1=\pi\circ f$ and $\dim\cM ( \cVnK ,f)\leq\mu (g_1)=k-1$.
This gives the same upper bound in case $(a)$, but in case
$(b)$ we have $l\leq k-1$.
\end{ex}

\section{The groups $\VqG \neq \cVpA$, $\cVnK$, $\cVpK$: examples of
$G$-stable maps $f$ of positive and infinite $\VqG$-modality}

In this final section we make some remarks on the
remaining volume preserving subgroups $\VqG$ of $\cA$ or $\cK$.
First of all we remark that placing volume forms both
in the source and the target of a map $f$ leads to moduli
even for invertible linear maps $f:\bC ^n\to\bC ^n$
(the modulus being the determinant of $f$).

For function-germs the only relevant groups are those
with a volume form to be preserved in the source,
and what is known for these had been described in
Section \ref{lit}.

For map-germs $\cR$-equivalence is too fine already in the
absence of a volume form, hence the remaining cases of
interest (not considered in the previous sections) are
the groups $\cVnA$, $\cVpL$ and $\cVpC$ for pairs of
dimensions $(n,p)$, $p>1$, for which singular
$G$-finite ($G=\cA$, $\cL$ or $\cC$) map-germs $f$
exist. And we can also discard those map-germs $f$
that are trivially w.q.h. for the relevant group.

The following (in some sense ``simplest'' singular but
non-w.q.h.) examples indicate that for the above three
groups we immediately obtain moduli.

\begin{ex}
For $\cVnA$ the fold map $f=(x,y^2)$
has infinite modality. We have
\[
L\cA _f^n=\bK \{ (x^ly^{2k},0), (0,x^ly^{2k+1}); l,k\ge 0, l+k\ge 1\},
\]
where the elements of $L\cA _f^n$ are also known as lowerable
vector fields (we write these source vector fields
as vectors). It follows that dimension of $C_n/\Div (L\cA _f^n)$,
which is a lower bound for the number of $\cVnA$-moduli,
is infinite for the fold $f$.
\end{ex}

\begin{ex}
For $\cVpL$ perhaps the first interesting example
of a singular germ that fails to be trivially w.q.h.
is the planar cusp $f=(x^2,x^3)$. We claim that in this case
$C_p/\Div (L\cL _f^p)\cong\bK\{ 1,y_1\}$, hence
the $\cVpL$-modality of $f$ is two ($f$ is $\cL$-simple
and the dimension of the $\cVpL$-moduli space is two).

Taking coordinates $(y_1,y_2)$ in the target, we see
that the kernel of
\[
L\cL\longrightarrow \cM _n\theta _f,~~~u\mapsto u\circ f
\]
is (as a $\bK$-vector space) generated by elements
$u^i_{rsl}:=y_1^ry_2^s(y_2^{2l}-y_1^{3l})\partial /\partial y_i$,
where $i=1,2$, $r,s\ge 0$ and $l\ge 1$.
Set $G^i_{rsl}:=\Div (u^i_{rsl})$, then
\[
(2l+s+1)G^1_{r+1,s,l}-(r+1)G^2_{r,s+1,l}=cy_1^{3l+r}y_2^s
\]
and
\[
(s+1)G^1_{r+1,s,l}-(3l+r+1)G^2_{r,s+1,l}=cy_1^ry_2^{2l+s}
\]
where $c=-6l^2-2l(r+1)-3l(s+1)\neq 0$. Finally, we have
$G^1_{001}=-3y_1^2$, $G^2_{001}=2y_2$, $G^1_{011}=-3y_1^2y_2$
and $G^2_{101}=2y_1y_2$, and the claim follows.
\end{ex}

\begin{ex}
For $\cVpC$ we first remark that $\cC$-finite germs $f$
can only appear for $n\le p$. As an example for a  singular
germ $f$, which fails to be trivially w.q.h., we can consider
the fold $f=(x,y^2)$. A quick calculation yields
$C_n/{\Div (L\cC _f^p)}\cong\bK\{ 1,y\}$. Hence $f$
has two $\cVpC$-moduli, which can also be checked by
comparing the normal spaces for $\cC$ and $\cVpC$.
Notice that $N\cC\cdot f$ is spanned by $(0,y)$ and $(y,0)$,
whereas $N\cVpC\cdot f$ is spanned by these two elements
together with $(x,0)$ and $(xy,0)$.
\end{ex}

{\bf Acknowledgements.} The authors wish to express their thanks to
Goo Ishikawa, Stanislaw Janeczko, Andrew du Plessis,
Maria Ruas and Michail Zhitomirskii for helpful conversations
about this work.

\end{document}